\documentclass[a4paper]{amsart}

\usepackage{amsmath, amssymb, amsthm}
\usepackage{url}
\usepackage[hidelinks]{hyperref}
\usepackage[hyphenbreaks]{breakurl}
\usepackage[capitalise, noabbrev]{cleveref}
\usepackage{bbm}
\usepackage{comment}
\usepackage[shortlabels]{enumitem}  
\usepackage{esint, bigints}  
\usepackage{floatrow}
\usepackage{fullpage}
\usepackage{mathrsfs}
\usepackage{mathtools}
\usepackage{setspace}
\usepackage{stmaryrd}
\usepackage[final,nopatch=footnote]{microtype}
\SetSymbolFont{stmry}{bold}{U}{stmry}{m}{n}
\usepackage{bm}  
\usepackage{tikz, tikz-cd, tikz-3dplot}
\usepackage{xcolor}

\newcommand*{\const}{\mathrm{const}}
\newcommand*{\id}{\mathrm{id}}
\newcommand*{\vol}{\mathrm{vol}}

\newcommand*{\CP}{\mathbb{CP}}
\newcommand*{\C}{\mathbb{C}}

\newcommand*{\Quat}{\mathbb{H}}
\newcommand*{\N}{\mathbb{N}}

\newcommand*{\R}{\mathbb{R}}

\renewcommand*{\S}{\mathbb{S}}
\newcommand*{\T}{\mathbb{T}}
\newcommand*{\Z}{\mathbb{Z}}

\newcommand*{\Mass}{\mathbf{M}}

\newcommand*{\SO}{\mathsf{SO}}

\newcommand*{\Cont}{\mathcal{C}}

\newcommand*{\Haus}{\mathcal{H}}

\renewcommand*{\Re}{\operatorname{Re}}
\newcommand*{\dd}{\mathrm{d}}
\newcommand*{\Dd}{\mathrm{D}}
\newcommand*{\del}{\partial}
\newcommand*{\eps}{\varepsilon}

\newcommand*{\into}{\hookrightarrow}

\newcommand*{\weakto}{\rightharpoonup}
\newcommand*{\Grass}{\Gr_2^+(\R^4)}
\newcommand*{\TAC}{\mathcal{T}}
\newcommand*{\AG}{\mathcal{AG}}
\newcommand{\extpow}{\textstyle{\bigwedge}}

\DeclareMathOperator{\Area}{Area}

\DeclareMathOperator{\Gr}{Gr}

\DeclareMathOperator{\Jac}{Jac}

\DeclareMathOperator{\dist}{dist}

\DeclareMathOperator{\dvol}{dvol}
\DeclareMathOperator{\graph}{graph}

\DeclareMathOperator{\res}{\mathbin{\mbox{\Large$\llcorner$}}}

\DeclareMathOperator{\spt}{spt}
\DeclareMathOperator{\Span}{span}

\DeclarePairedDelimiter\DBrack{\llbracket}{\rrbracket}

\newtheorem{thm}{Theorem}[section]
\newtheorem{cor}[thm]{Corollary}
\newtheorem{lemma}[thm]{Lemma}

\newtheorem{prop}[thm]{Proposition}

\theoremstyle{definition}
\newtheorem{defn}[thm]{Definition}

\newtheorem{rmk}[thm]{Remark}

\makeatletter
\renewcommand\th@plain{\slshape}
\makeatother

\numberwithin{equation}{section}

\newlist{steps}{enumerate}{1}
\setlist[steps, 1]{label={\underline{{Step} \arabic*.}}, leftmargin=*, wide, labelindent=0pt, itemsep=5pt}

\usepackage[backend=biber,style=numeric,sorting=nyt,doi=false,isbn=false,url=true,eprint=true,maxbibnames=99]{biblatex}
\usepackage{xurl}
\usepackage{doi}
\renewbibmacro{in:}{}

\AtEveryBibitem{%
\ifentrytype{article}{
    \clearfield{url}%
    \clearfield{urldate}%
}{}
\ifentrytype{book}{
    \clearfield{url}%
    \clearfield{urldate}%
}{}
\ifentrytype{collection}{
    \clearfield{url}%
    \clearfield{urldate}%
}{}
\ifentrytype{incollection}{
    \clearfield{url}%
    \clearfield{urldate}%
}{}
}

\DeclareFieldFormat{edition}{#1 edition}

\author{Gerard Orriols}
\address{ETH, Rämistrasse 101, 8092 Zürich, Switzerland}
\email{gerard.orriols@math.ethz.ch}

\author{Tristan Rivi\`ere}
\address{ETH, Rämistrasse 101, 8092 Zürich, Switzerland}
\email{tristan.riviere@math.ethz.ch}

\date{June 5, 2025}

\title{Minimizing the Gauss map area of surfaces in $\S^3$}
\begin{document}
\maketitle
\begin{abstract}
  We establish the lower bound of $4\pi(1+g)$ for the area of the Gauss map of any immersion of a closed oriented surface of genus $g$ into $\S^3$,
  taking values in the Grassmannian of $2$-planes in $\R^4$.
  This lower bound is proved to be optimal for any genus $g \in {\N}$ but attained only when $g=0$. For $g \neq 0$ we describe the behavior of any minimizing sequence of embeddings: we prove that, modulo extraction of a subsequence, the surfaces converge in the Hausdorff distance to a round sphere $S$, and the integral cycles carried by the Gauss maps split into $g+1$ spheres, each of area $4\pi$; one of them corresponds to the cycle carried by the Gauss map of $S$, while the other $g$ arise from the concentration of negative Gauss curvature at $g$ points of $S$.
  
The results of this paper are used by the second author to define a nontrivial homological 4-dimensional min-max scheme for the area of Gauss maps of immersions into ${\mathbb S}^3$ in relation to the Willmore conjecture.
\end{abstract}

\section{Introduction}

Let $\Sigma$ be a closed oriented surface of genus $g \geq 0$ and $\phi : \Sigma \to \R^3$ be an immersion with Gauss (or unit normal) map $\nu : \Sigma \to \S^2$. The area covered by $\nu$ is given by the integral of the total curvature of this immersion:
\[
  \Area(\nu) = \int_{\Sigma} |K_\phi| \dvol_\phi,
\]
where $K_\phi$ is the Gauss curvature defined by
\[
  \nu^* \vol_{\S^2} = K_\phi \vol_{\phi},
\]
and $\vol_{\S^2}$ is the standard volume form on $\S^2$. A well known inequality going back at least to Chern and Lashof \cite{chern-lashof2} asserts that for any such immersion,
\begin{equation}
  \label{int-1}
  \int_{\Sigma}|K_\phi| \, \dvol_\phi \ge 4\pi\,(1+g),
\end{equation}
and moreover, for any genus $g$, equality in \eqref{int-1} holds for a large class of embeddings, which later became known as \emph{tight} (see \cite{cecil-ryan} for an introduction on the subject). A classical proof of \eqref{int-1} goes through the introduction of the function $h_e : x \in \Sigma \mapsto \phi(x) \cdot e \in \R$ for a fixed choice of $e \in {\mathbb S}^2$, called ``height function''. Such a function $h_e$ happens to be Morse for almost every $e\in \S^2$, and thus by the classical Morse inequalities (see for instance Section~5 in \cite{milnor-morse-theory}), it has at least $2g+2$ critical points, which are the preimages of $\pm e$ by $\nu$. One then concludes by the coarea formula.

Another (related) approach to this inequality is due to Willmore \cite{willmore-conjecture} and consists in first showing that
\begin{equation}
  \label{int-2}
  \int_{\Sigma} K_\phi^+ \, \dvol_\phi \geq 4\pi,
\end{equation}
where $K_\phi^+$ is the positive part of the Gauss curvature. Combining \eqref{int-2} with the Gauss--Bonnet identity one immediatly obtains the Chern--Lashof inequality \eqref{int-1}. The reason why \eqref{int-2} is true is quite geometric: by translating parallelly any plane coming from infinity one must hit for the first time the immersion at a point where $K_\phi = \nu_\phi^\ast \vol_{\S^2}\ge 0$. Hence for any $e\in \S^2$ there is at least one point $x\in \Sigma$ at which $\nu(x)=e$ and $K_\phi(x) \ge 0$, and one again concludes thanks to the coarea formula.

This argument shows easily that equality holds in \eqref{int-2}, and hence in \eqref{int-1}, if and only if the set where $K_\phi > 0$ is contained in the boundary of the convex hull of $\phi(\Sigma)$. The map $\nu$ restricted to these points covers $\S^2$ positively and exactly once, and on the rest of $\Sigma$ it covers $\S^2$ negatively and exactly $g$ times.

The goal of the present paper is to extend this result to immersions into $\S^3$ and to prove a \emph{rigidity theorem} for surfaces that (almost) saturate the inequality, which extends the decomposition into ``positive'' and ``negative spheres'' mentioned above but in a geometrically more involved setting.

\medskip

For surface immersions into ${\R}^n$ where $n$ is arbitrary (i.e.~in codimension $n-2$), Chern and Lashof gave a counterpart to the inequality \eqref{int-1}. More precisely, they considered the unit sphere bundle $U\Sigma$ inside the normal bundle of $\phi(\Sigma)$ in ${\R}^n$,
\[
  U\Sigma:=\left\{ (x,\xi)\in \Sigma\times {\mathbb S}^{n-1} \, : \, \xi\in (\phi_\ast T_x\Sigma)^\perp \right\}.
\]
This defines an $\S^{n-3}$-bundle over $\Sigma$ and thus $U\Sigma$ is a closed $(n-1)$-dimensional manifold, equipped with a tautological map
\begin{equation}
  \label{eq:def-F}
  F : (x, \xi) \in U\Sigma \longmapsto  \xi \in {\mathbb S}^{n-1}.
\end{equation}
A computation using the Gauss--Bonnet theorem shows that
\begin{equation}
  \label{int-2-b}
  \int_{U\Sigma} F^\ast \vol_{\S^{n-1}}=|{\mathbb S}^{n-1}|\, (2-2g),
\end{equation}
or in other words, the topological degree of $F$ equals $2-2g$. Chern and Lashof consider the absolute value of the area of this map $F$ given by
\[
  \TAC(\phi):=\int_{U\Sigma} |F^\ast \vol_{{\mathbb S}^{n-1}}| \, \dvol_{U\Sigma}=\int_{U\Sigma} \Jac(F) \, \dvol_{U\Sigma}\ ,
\]
where $ \Jac(F)$ is the Jacobian of $F$, and they prove the following lower bound\footnote{Observe that for $n=3$, for an oriented closed surface $\Sigma$ of genus $g$, $U\Sigma\simeq \Sigma\times \S^0$ is a double cover of $\Sigma$ locally isometric to $\Sigma$ and therefore
\[
\int_{U\Sigma} |F^\ast \vol_{\S^{n-1}}|
= 2\, \int_\Sigma |\nu^\ast \vol_{\S^2}|
.
\]  } for any immersion $\phi$ of a closed oriented surface $\Sigma$ of genus $g$:
\begin{equation}
\label{int-3}
\TAC(\phi)\ge |{\mathbb S}^{n-1}|\, (2 g+2)
\end{equation}
A proof of this inequality generalizing Willmore's approach to arbitrary codimension is given in \cite[Theorem X.2]{riviere-conformally-invariant-problems}. It amounts to showing that here again there must be a double positive covering by $F$ of the unit sphere ${\mathbb S}^{n-1}$. The result then follows from \eqref{int-2-b}.

\medskip

In the present paper we shall focus on the case where $\phi$ is an immersion into the 3-sphere ${\mathbb S}^3$ and denote by $\nu$ the unit normal vector map $\nu(x) \in T_{\phi(x)}{\mathbb S}^3\cap (\phi_\ast T_x\Sigma)^\perp$ compatible with the orientation of $\Sigma$. The Gauss map $G_\phi$ assigns to any $x\in \Sigma$ the oriented $2$-plane in $\R^4$ \emph{normal}\footnote{This is essentially equivalent to considering the tangent planes, but notationally more convenient.} to the immersed surface at $\phi(x)$. Explicitly, denoting the space of oriented $2$-planes in ${\R}^4$ by $\Gr_2^+({\R}^4)$, we have
\[
  G_\phi : x\in\Sigma \longmapsto G_\phi(x) := \phi(x) \wedge \nu(x) \in \Grass.
\]
The space $\Gr_2^+({\R}^4)$ of oriented $2$-planes in ${\R}^4$ (i.e.~the space of unit simple $2$-vectors) can be identified with $ \tfrac{1}{\sqrt{2}}\,({\mathbb S}^2\times{\mathbb S}^2)\subset {\mathbb S}^5$ through the following isometry:
\begin{equation}
  \label{eq:quaternion-coordinates}
  I : a\wedge b\in \Gr_2^+({\R}^4)\ \longmapsto \ 
  \frac{1}{\sqrt{2}}\,\left(
    \begin{pmatrix}
      \langle \bm{i} \bm{a}, \bm{b} \rangle \\
      \langle \bm{j} \bm{a}, \bm{b} \rangle \\
      \langle \bm{k} \bm{a}, \bm{b} \rangle
    \end{pmatrix}
    ,
    \begin{pmatrix}
      \langle \bm{a} \bm{i}, \bm{b} \rangle \\
      \langle \bm{a} \bm{j}, \bm{b} \rangle \\
      \langle \bm{a} \bm{k}, \bm{b} \rangle
    \end{pmatrix}
  \right)
  \in \S^2 \left(\tfrac{1}{\sqrt{2}} \right) \times \S^2 \left(\tfrac{1}{\sqrt{2}} \right).
\end{equation}
where $\bm{a}$ and $\bm{b}$ are respectively the quaternionic representatives of $a$ and $b$ (see \cref{sec:prelim}). This identification gives a convenient way of writing the image of the Gauss map of a round sphere (see \cref{prop:graph-rotation-plane}): if $S = \del B_r(p) \subset \S^3$ for some $p \in \S^3$ and $0 < r < \pi$, then
\[
  \Gamma_p := G_S(S) = \{ p \wedge v : v \in \S^3 \cap p^\perp \} \subset \Grass
\]
is identified via $I$ with
\begin{equation}
  \label{eq:gauss-map-round-sphere}
  I(\Gamma_p) = \graph(R_p) = \left\{ (z, R_p z) : z \in \S^2 \left(\tfrac{1}{\sqrt{2}} \right) \right\}
  \subset \S^2 \left(\tfrac{1}{\sqrt{2}} \right) \times \S^2 \left(\tfrac{1}{\sqrt{2}} \right),
\end{equation}
where $R_p \in \SO(3)$ is given by the double cover $p \in \S^3 \to R_p \in \SO(3)$ made explicit in \eqref{eq:rotation-quaternions}.

The area of $G_\phi$ as a map from $\Sigma$ into $\Grass$ is then given by
\[
  \AG(\phi):=\Area(G_\phi)=\int_\Sigma\sqrt{(1+\kappa_1^2)(1+\kappa_2^2)} \, \dvol_\phi,
\]
where $\kappa_1$ and $\kappa_2$ are the principal curvatures of $\phi(\Sigma)$, that is, the eigenvalues of the second fundamental form of the immersion $\phi$ viewed as a codimension one immersion into ${\mathbb S}^3$.

\subsection{Main results}
Our first result is an optimal lower bound for the area of $G_\phi$ for any genus:
\begin{thm}
\label{thm:area-gauss-bound}
Let $\phi$ be an immersion of a closed oriented surface $\Sigma$ of genus $g$ into $\S^3$. Then the following inequality holds:
\begin{equation}
  \label{eq:area-tac-genus-bound}
  \AG(\phi) = \Area(G_\phi) \ge \frac{\TAC(\phi)}{\pi},
\end{equation}
and in particular
\begin{equation}
  \label{eq:area-gauss-bound}
  \Area(G_\phi) \ge 4\pi (1+g).
\end{equation}
Equality in \eqref{eq:area-tac-genus-bound} (and thus also in \eqref{eq:area-gauss-bound}) is true if and only if $g=0$ (i.e.~$\Sigma \simeq {\mathbb S}^2$) and $\phi$ parametrizes a round sphere, that is, a set of the form $\del B_r(p)$. In this case, there exists a diffeomorphism $\iota : \S^2 \to \Sigma$ such that
\[
  I \circ G_\phi \circ \iota(z) = \frac{1}{\sqrt{2}} (z, R_p(z)) \qquad \forall z \in \S^2.
\]
\end{thm}

Even though equality is only attained for $g = 0$, the bound is optimal for every genus. This can be easily seen by transplanting tight surfaces from $\R^3$ into a neighborhood of a point in $\S^3$ (see \cref{prop:concentrating-sequence}), but we can also have a more interesting behavior:
\begin{thm}
  \label{thm:concentration-handles}
  Given an integer $g > 0$ and $g$ points $p_1, \ldots, p_{g}$ that lie on a round $2$-sphere $S = \del B_r(p) \subset \S^3$, there exists a sequence of smooth embedded surfaces $\Sigma_j \subset \S^3$ of genus $g$ for which the Gauss map area measure $\gamma_j := \sqrt{1 + \kappa_1^2}\sqrt{1 + \kappa_2^2} \, \Haus^2 \res \Sigma_j$ satisfies
  \begin{equation}
    \label{eq:concentration-gauss-measure}
    \gamma_j
    \weakto \gamma_S + \sum_{k=1}^{g} 4\pi \delta_{p_k}
    = 4 \pi \frac{\Haus^2 \res S}{\Haus^2(S)} + \sum_{k=1}^{g} 4\pi \delta_{p_k}.
  \end{equation}
  as $j \to \infty$ in the sense of measures. In particular $\AG(\Sigma_j) = \gamma_j(\S^3) \searrow 4\pi(g + 1)$, so $\{\Sigma_j\}$ is a minimizing sequence for $\AG$. Moreover the pushforward integral currents satisfy
  \begin{equation}
    \label{eq:convergence-currents-example}
    (I \circ G)_*\DBrack{\Sigma_j} \weakto \DBrack{\graph(R_{p_0})} - \sum_{k=1}^{g} \, \DBrack{\graph(R_{p_k})}.
  \end{equation}
\end{thm}

Observe that \eqref{eq:convergence-currents-example} could be interpreted as the $\S^3$ counterpart of the convergence to a \emph{tight} immersion for minimizing sequences of the Gauss map area of immersions into ${\R}^3$. In that case, exactly one positive bubble and $g$ negative ones are formed at the limit. 

The second main result of the present work shows that the behavior described in \cref{thm:concentration-handles} holds for general minimizing sequences. This gives a positive answer to the question formulated in \cite[Open Problem IV.1]{riviere-minmax-hierarchies}, assuming that the minimizing sequence consists of embeddings. More precisely:

\begin{thm}
  \label{thm:gauss-main}
  Let $\phi_j : \Sigma \to \S^3$ be a sequence of smooth immersions of a closed oriented surface $\Sigma$ of genus $g$, with normal maps $\nu_j : \Sigma \to \S^3$ and Gauss maps $G_j : \Sigma \to \Grass$. Suppose that $\AG(\phi_j) = \Area(G_j) \to 4\pi(1+g)$, which is the minimum in this class.

  Then, after passing to a subsequence, there exists a (possibly degenerate) round sphere $S \subset \S^3$ such that $\dist(\phi_j(\cdot), S) \to 0$ uniformly. If in addition the maps $\phi_j$ are embeddings, then $\phi_j(\Sigma) \to S$ in the Hausdorff distance
  and there exist $g$ points $p_1, \ldots, p_g \in S$ such that
  \begin{equation}
    \label{eq:convergence-cycles-1}
    (G_j)_\# \DBrack{\Sigma} \xrightharpoonup{\; j \to \infty \;} \DBrack{\Gamma_{p_0}} - \DBrack{\Gamma_{p_g}} - \cdots - \DBrack{\Gamma_{p_g}} 
  \end{equation}
  or equivalently
  \begin{equation}
    \label{eq:convergence-cycles-2}
    (I \circ G_j)_\# \DBrack{\Sigma} \xrightharpoonup{\; j \to \infty \;} \DBrack{\graph(R_{p_0})} - \DBrack{\graph(R_{p_1})} - \cdots - \DBrack{\graph(R_{p_g})}
  \end{equation}
  as integral cycles in $\S^2_+(1/\sqrt{2}) \times \S^2_-(1/\sqrt{2})$, where $p_0$ is the center of $S$.
  Moreover we have convergence of the corresponding integral varifolds in $\Grass$:
  \begin{equation}
    \label{eq:convergence-varifolds}
    \mathbf{v}(G_j(\Sigma))
    \xrightharpoonup{\; j \to \infty \;}
    \mathbf{v}(\Gamma_{p_0}) + \mathbf{v}(\Gamma_{p_1}) + \cdots + \mathbf{v}(\Gamma_{p_g})
  \end{equation}
  and measures on $\S^3$:
  \begin{equation}
    \label{eq:convergence-measures-s3}
    \gamma_j := \sqrt{1 + \kappa_1^2}\sqrt{1 + \kappa_2^2} \, \Haus^2 \res \phi_j(\Sigma)
    \xrightharpoonup{\; j \to \infty \;}
    \gamma_S + 4\pi \delta_{p_1} + \cdots + 4\pi \delta_{p_g}.
  \end{equation}
\end{thm}

\cref{thm:gauss-main} can be interpreted as a qualitative stability result accompanying \cref{thm:area-gauss-bound} (and indeed can be rewritten as a quantitative $\eps$-$\delta$ statement for the flat and $\mathbf{F}$ distance by using a standard compactness argument). 
The study of the qualitative and quantitative stability of optimal geometric inequalities in which the equality cases are well understood is a common and modern theme in Geometric Analysis---see for example \cite{fusco-maggi-pratelli, figalli-maggi-pratelli, de-lellis-muller1, lamm-schatzle-stability, inauen-marchese-stable-quantitative, de-philippis-maggi-stability-plateau, bernandmantel-muratov-simon, rupflin-quant-rigidity-degree, luckhaus-zemas} for a few recent instances.

In many of the cases mentioned above there are two ingredients that makes a qualitative stability result trivial and a sharp quantitative analysis possible: the existence of minimizers, and an \emph{a priori} compactness theorem that ensures that minimizing sequences converge to a minimizer. Instead, in our case (for $g > 0$) there are no minimizers at all, and general compactness theorems for the Gauss maps $G_j$ under just a control on the mass, such as the Federer--Fleming compactness theorem for the integral cycles $(G_j)_\# \DBrack{\Sigma}$, provide weak limits which escape from the class of objects for which we can prove the inequality (e.g.~for $g=1$ they could converge to zero as integral currents). Thus the main obstacle to overcome is obtaining some compactness for the minimizing sequence, which must come from a good understanding of the geometry of the problem. An added difficulty comes from the fact that this is a vectorial problem and the Gauss maps take values in a space of codimension two, so one has to work with weaker geometric objects and norms in comparison to other settings.

Regarding the optimality of our result, first observe that very little can be said about the immersion $\phi$ itself apart from the Hausdorff convergence to a sphere or a point. The convergence at the level of the Gauss map is more interesting because of the behavior stated in \cref{thm:gauss-main}, and more natural because it recovers all the mass at the limit, as with more classical bubbling phenomena. If one asks about improving the rate of convergence instead of the chosen norm, we do not know about the optimal one, but it is a priori not clear, since the limiting object is degenerate and singular.

\begin{rmk}
  The embeddedness assumption is necessary to obtain such a precise characterization of the limiting object: for example, it is easy to construct immersions $\phi : \T^2 \to \R^3$ of tori with $\TAC(\phi)$ arbitrarily close to $8\pi$, for which the normal map $\nu$ takes values in an arbitrarily small neighborhood of a hemisphere of $\S^2$. These look like the torus of revolution of a very thin figure-8 curve. Composing such maps $\phi$ with a dilation by a small factor and then the inverse stereographic projection $\R^3 \to \S^3$, one obtains minimizing sequences for $\AG$ consisting of immersions of tori such that the right hand side of \eqref{eq:convergence-varifolds} has, instead of $\mathbf{v}(\Gamma_{p_k})$, the varifold supported in a hemisphere of $\Gamma_{p_k}$ with multiplicity two.
\end{rmk}

To get an intuition of the contents of the theorems it may be useful to think about the case of curves in $\S^2$, where the analysis is much easier if one works with arc-length parametrizations. In particular, the lower bound for the length of the (tangent) Gauss map $G = \phi \times \nu$ of an immersion $\phi : \S^1 \to \S^2$ with unit normal $\nu$ and curvature $\kappa$ is given by $\int_{\S^1} \sqrt{1+\kappa^2} \dd \Haus^1 \ge 2\pi$, thanks to a classical inequality of Fenchel \cite{fenchel}, and equality holds if and only if $\phi(\S^1)$ is plane and convex, therefore a round circle. One can also see with some work that minimizing sequences of embedded curves must either converge to a round circle or concentrate towards a point, and $G_j(\S^1)$ converges in a suitable sense to an equatorial $\S^1 \subset \S^2$.

Finally we would like to remark that the analysis developed here serves as one of the first steps (to our knowledge) into a rigidity and stability analysis in the theory of \emph{tight} and \emph{taut} surfaces, which are minimizers of the total absolute curvature $\TAC$ among surfaces in $\R^3$ and $\S^3$, respectively. These are very rich classes of submanifolds and have been studied by many authors in the second half of the last century (see for example \cite{chern-lashof2, banchoff, pinkall-taut-dupin, kuiper-tight} and the book \cite{cecil-ryan}). Most of the literature of the subject is restricted to topological and geometric properties of \emph{minimizers}, and here we study \emph{almost} minimizers of the functional $\AG$, which is closely related to $\TAC$. It would be very interesting to understand when and in which sense almost minimizers of $\TAC$ are close to Dupin surfaces or spheres.

\subsection{Main ideas of the proof}
The first step towards the proof of \cref{thm:gauss-main} is a quantitative stability estimate that leads to the subsequential uniform convergence of the immersions to a round sphere.
For any $x\in \Sigma$, let $l(x)$ [resp.~$\tilde l(x)$] denote the length of the longest geodesic that intersects $\phi(\Sigma)$ orthogonally at $\phi(x)$ and such that the segments before and after $\phi(x)$ are minimizing [resp.~stable] as geodesics with one end in $\phi(\Sigma)$. It is easy to see that $0 < l(x) \leq \tilde l(x) \leq \pi$ for every $x$, and moreover, if $l(x) = \pi$ for some $x$, then $\Sigma$ is a round sphere.

In order to prove \cref{thm:gauss-main} our strategy consists in finding a sequence of points $x_j \in \Sigma$ such that  
\begin{equation}
  \label{int-7b}
  l(x_j) \rightarrow \pi. 
\end{equation}
The existence of this sequence comes from the following quantitative stability inequalities (see \cref{lem:estimates-Sigma-pi} and \cref{cor:convergence-round-sphere}): for each $g>0$ there exists $C = C(g)>0$ such that for any $0<\delta<1$, if
\[
  \Area(G_\phi) \le 4\pi (1+g)+\delta,
\]
then
\begin{equation}
  \label{int-8}
  \int_{\left\{ 0 < \tilde{l} < \pi/2 \right\}} \sqrt{(1+\kappa_1^2)(1+\kappa_2^2)} \, \dvol_\phi \ge 4\pi - C \delta
\end{equation}
and
\begin{equation}
  \label{int-9}
  \int_{\left\{ 0 < \tilde{l} < \pi/2 \right\}} (\pi - l)\sqrt{(1+\kappa_1^2)(1+\kappa_2^2)} \, \dvol_\phi \le C \delta^{1/3}
\end{equation}
(we do not know whether the exponent in \eqref{int-9} is optimal). By combining \eqref{int-8} and \eqref{int-9} we obtain a sequence of points such that \eqref{int-7b} holds, and therefore $\phi_j(\Sigma)$ converges to a possibly degenerate round 2-sphere in ${\mathbb S}^3$.

Then the challenge is to separate the Gauss map of the surface into a macroscopic positive curvature region, where the Gauss map covers one bubble of area approximately $4\pi$, and highly concentrated negative curvature regions, whose Gauss map form $g$ bubbles, each of area nearly $4\pi$ as well. In order to be able to apply the Federer--Fleming compactness theorem, we need to produce \emph{closed} integral currents that capture the image of the Gauss map restricted to these regions.

Recall that, for almost tight surfaces $\Sigma \subset \R^3$, we detected most of the positive curvature region as the contact set with planes coming from infinity in each direction; the set of the normal vectors to these planes forms a ``positive'' $\S^2$ that corresponds to the image by $\nu$ of the convex hull of the surface. The restriction of $\nu$ to the rest of the surface covers an area of $\int_\Sigma |K| - 4\pi$, which approaches $4\pi g$, and at the limit yields a current $g \DBrack{\S^2}$ with the negative orientation.

Unfortunately we cannot generalize this construction to surfaces in $\S^3$, because there is no satisfactory analog of the method of moving planes or of the convex hull. A related fact is that, in general, for a surface of positive genus in $\S^3$ the integral of the positive part of the Gauss curvature may be arbitrarily small (and in fact there are tori, like the Clifford torus, with zero Gauss curvature everywhere).

Nevertheless, there is a similar argument for $\Sigma \subset \R^3$ that we can mimic in $\S^3$. Consider the fattened set $\Sigma_{(d)} := \{p : \dist(p, \Sigma) = d \} \subset \R^3$ for $d > 0$. When $d$ is very large, $\Sigma_{(d)}$ is topologically a sphere, its total absolute curvature approaches $4\pi$, and the closest projection map $\pi$ from $\Sigma_{(d)}$ to $\Sigma$ preserves the normal map; in particular, $\nu$ restricted to $\pi(\Sigma_{(d)})$ defines a cycle with area close to $4\pi$ that captures most of the positive curvature of $\Sigma$.
The set $\pi(\Sigma_{(d)}) \subset \Sigma$ can also be characterized as the set of points $x \in \Sigma$ for which every point $x + t \nu(x)$ for $0 \leq t \leq d$ is at distance $t$ from $\Sigma$, and ``regularizes'' $\Sigma$ in the sense that it gets rid of the regions with very negative Gauss curvature.

Thanks to the geometric estimates sketched above, this construction can be adapted to surfaces in $\S^3$ and as a result we get the separate convergence of the ``positive'' and ``negative'' parts of the integral cycle carried by the Gauss map. The characterization of the limiting current as the sum of graphs of rotations comes from a short calibration argument.

\subsection{Application to min-max for Gauss maps and the Willmore problem}
In \cite{riviere-minmax-hierarchies} the second author conjectured that, modulo extraction of a subsequence, \eqref{eq:convergence-cycles-2} always hold true for \emph{any} minimizing sequence of $\Area(G_\phi)$ and fixed $\Sigma$. This question has been posed in relation to a new proof of the Willmore conjecture (after \cite{marques-neves-willmore}) via a min-max operation on the area of the Gauss map. More precisely, assuming  \eqref{eq:convergence-cycles-2} always holds true for {any} minimizing sequence of $\Area(G_\phi)$ one defines for any immersion of $\Sigma$ satisfying
\[
  \Area(G_{\phi}) \le 4\pi (1+g) + \eps
\]
and one can define the map $\mathcal P$, called \emph{polarization map}, which assigns the nearest positive bubble $R_{p_0} \in \SO(3)$ to $G_\phi$ appearing in \eqref{eq:convergence-cycles-2} (see \cite{riviere-minmax-hierarchies}). One then introduces the admissible family given by
\[
{\mathcal A}:=\left\{
\begin{array}{c}
 \Phi\in C^0_{\bf F}(\mbox{int}(X^4),\mbox{Imm}(\Sigma, {\mathbb S}^3))\ ; \ \exists \, G\in C^0_{\bf F}(X^4, \mbox{Gr}_2^+({\R}^4))\\[3mm]
 X^4\ \mbox{ is a  compact 4-dimensional polyhedron with boundary }\\[3mm]
 G=G_\Phi\quad\mbox{ in }\mbox{int}(X^4)\quad,\quad \mbox{Area }(G)\equiv 4\pi\, (1+g)\mbox{ on }\del X^4\\[3mm]
 ({\mathcal P}\circ G)_\ast\del X^4\ne 0\quad\mbox{ in }H_3(\SO(3),{\Z})
 \end{array}
 \right\}
\]
where $\bf F$ is the varifold distance. It is proved in \cite{riviere-minmax-hierarchies} that for any $g>0$ $\mathcal A\ne \varnothing$. Assuming then that \eqref{eq:convergence-cycles-2} always hold true for {any} minimizing sequence of $\Area(G_\phi)$  one can deduce that
\[
  \mbox{Width}_g:=\inf_{\Phi\in {\mathcal A}} \max \Area(G_\Phi)>4\pi\, (1+g)
\]
which makes the min-max operation well defined.

\subsection{Gauss maps in symplectic geometry}
\label{sec:gauss-maps-lagrangian}
An important motivation for us to study Gauss maps of genus $g$ surfaces in $\S^3$ is that, for $\phi$ in a regular homotopy class of immersions, the maps $G_\phi$ form an open subset (in the smooth topology) of an exact regular homotopy class of Lagrangian surfaces in $\S^2 \times \S^2$. Alternatively, one can see Lagrangian surfaces in this class as generalizations of Gauss maps of surfaces in $\S^3$, for which the immersion $\phi$ may become singular at some points but the normal map $\nu$ does not\footnote{These maps are sometimes called \emph{Lie immersions} in the literature, e.g.~in \cite{palmer-buckling}.}. We hope that this point of view, at least going back to \cite{palmer-buckling, palmer-hamiltonian-gauss-maps}, provides a concrete framework, more accessible from the point of view of classical differential geometry, to the study of these classes of surfaces, in particular with regard to metric questions like the area. See for example \cite{hind-spheres-s2-s2, torus-minimizer-japan, drgi-isotopy-s2-s2} for related questions.

Arbitrary Lagrangian maps $G : \Sigma \to \Grass$ are in general very far from being Gauss maps. A first necessary condition is that they admit a Legendrian lift into $V_2(\R^4)$, the Stiefel manifold of orthonormal $2$-frames in $\R^4$, which is a contact $5$-dimensional manifold forming a principal $\mathsf{U}(1)$-bundle $\pi : V_2(\R^4) \to \Grass$ with curvature $\omega$. The existence of this lift can be characterized by the fact that all the periods $\int_D \omega$ are multiples of $2\pi$, where $D$ is any immersed disk in $\Grass$ with boundary in $G(\Sigma)$.

Notice that this is a global property: a Lagrangian map $G$ is always locally exact and thus always admits a local Legendrian lift $\Lambda$. Furthermore, the area measures of $G$ and of $\Lambda$ agree, since the projection $\pi : V_2(\R^4) \to \Grass$ induces a linear isometry on the horizontal planes.

Now if $(a, b) : U \to V_2(\R^4)$ is a local lift of $G$ and $a$ happens to be an immersion into $\S^3$ (or equivalently the $2$-form $\star (\dd a \wedge \dd a \wedge a \wedge b)$ does not vanish), then $G$ is the Gauss map of $a$ with unit normal $b$. Therefore, provided that $G$ does admit a global Legendrian lift $(a, b)$, $G$ will be a Gauss map if and only if $a \cos \theta + b \sin \theta$ is an immersion for some constant $\theta$. This is not always the case, and in fact it is easy to construct Legendrian maps which do not obey \cref{thm:area-gauss-bound} and hence do not come from immersions---see \cite{orriols-phd} for more details.

We finish this subsection by noting that we do not know what is the lower bound for the area among Lagrangian immersions $\Sigma \to \Grass$ which are exact regular homotopic to the Gauss map $G_\phi$ of an immersion $\phi : \Sigma \to \S^3$.

\subsection{Structure of the paper}

In \cref{sec:prelim} we define the Gauss map and describe the relevant geometric properties of the Grassmannian of $2$-planes in $\R^4$, where it takes values. Next, in \cref{sec:area-gauss-tac}, we relate the area $\AG(\phi)$ of the Gauss map of an immersion $\phi$ to its total absolute curvature $\TAC(\phi)$ and prove \cref{thm:area-gauss-bound}. In \cref{sec:min-sequences} we discuss the optimality of the bound for $\AG$ and prove \cref{thm:concentration-handles}. Finally, in \cref{sec:compactness-gauss} we prove (a slightly stronger version of) \cref{thm:gauss-main}.

\section{Preliminaries}
\label{sec:prelim}

\subsection{The Grassmannian and Gauss maps}

Let $\Sigma$ be a closed orientable surface of genus $g \geq 0$, and let $\phi : \Sigma \to \S^3$ be a smooth immersion. We choose a global unit normal $\nu : \Sigma \to \S^3$ compatible with the orientation of $\Sigma$ and the standard orientation of $\S^3 \subset \R^4$, meaning that if $e_1, e_2$ is an oriented basis of $T_x \Sigma$, then $\{ \nu(x), \phi_\# e_1, \phi_\# e_2 \}$ is an oriented basis of $T_{\phi(x)} \S^3$ and thus $\{ \phi(x), \nu(x), \phi_\# e_1, \phi_\# e_2 \}$ is an oriented basis of $\R^4$.

Consider now $\phi$ as an immersion of $\Sigma$ into $\R^4$. We define its Gauss map as follows:
\begin{defn}
  \label{def:gauss-map}
  The \emph{(normal) Gauss map} of an oriented immersion $\phi : \Sigma \to \S^3$ is the map
  \begin{equation}
    \label{eq:def-gauss-map}
    x \in \Sigma \longmapsto G(x) = \phi(x) \wedge \nu(x) \in \Grass \subset \textstyle{\bigwedge^2} \left( \R^4 \right).
  \end{equation}
  Here we are identifying oriented planes in $\Grass$ with their orienting unit $2$-vectors in $\bigwedge^2 \R^4$, and under this definition, $G(x) = \star G_\text{tan}(x) = G_\text{tan}(x)^\perp$, where $G_\text{tan} : \Sigma \to \Grass$ is the tangent map (in local coordinates, $G_\text{tan}(x) = \partial_{x_1} \phi \wedge \partial_{x_2} \phi / |\partial_{x_1} \phi \wedge \partial_{x_2} \phi|$) and $\star$ is the Hodge star.
\end{defn}

Since $\star$ is an isometry, this is equivalent to the usual Gauss map, but this definition is much more convenient for computations. Moreover, this definition generalizes for surfaces without an orientation to a Gauss map on the unit normal bundle of $\Sigma$ inside $\S^3$, which is an oriented double cover of $\Sigma$; in fact, with the same definition, one also gets a Gauss map $G : \S(N \Sigma) \to \Gr_2^+(\R^{n+1})$ for immersions of surfaces $\phi : \Sigma \to \S^n \subset \R^{n+1}$, this time defined on the unit sphere bundle of the normal bundle $N \Sigma$ of $\Sigma$ in $\S^n$.

To proceed we will need some facts about the geometry of the Grassmannian.

\begin{prop}
  \label{prop:tangent-grass}
  The tangent space to $\Gr_2^+(\R^{n+1})$ at a point $\pi = a \wedge b$, where $|a| = |b| = 1$ and $a \cdot b = 0$, is
  \[
    T_{a \wedge b}\Gr_2^+(\R^{n+1}) = \left\{ a \wedge w + v \wedge b : v, w \in \pi^\perp \right\} \subset \extpow^2 \R^{n+1},
  \]
    and hence is isomorphic to $(\pi^\perp)^2 = \Span\{a, b\}^2$.
  \begin{proof}
  Given a curve $\pi(t) = a(t) \wedge b(t)$ such that $a(0) = a$, $b(0) = b$, and such that $a(t)$ and $b(t)$ are orthonormal for every $t$, we have that $ \left. \tfrac{\dd}{\dd t} \right|_{t=0} a(t) \wedge b(t) = a \wedge b'(0) + a'(0) \wedge b$. Let $\kappa := b'(0) \cdot a = -a'(0) \cdot b \in \R$ and define
    \[
      \tilde a(t) := a(t) \cos (\kappa t) + b(t) \sin (\kappa t), \qquad
      \tilde b(t) := -a(t) \sin (\kappa t) + b(t) \cos (\kappa t)
    \]
    so that $\tilde a(t)$ and $\tilde b(t)$ still form an orthonormal basis of $\pi(t)$ for all $t$. Then
    \[
      \tilde a'(0) = a'(0) + \kappa b
      \qquad \text{and} \qquad
      \tilde b'(0) = -\kappa a + b'(0),
    \]
    so it is clear that $\tilde a'(0) \cdot a = 0$, $\tilde b'(0) \cdot b = 0$, and in addition
    \[
      \tilde a'(0) \cdot b = a'(0) \cdot b + \kappa = 0
      \quad \text{and} \quad
      \tilde b'(0) \cdot a = -\kappa + b'(0) \cdot a = 0.
    \]
    This shows that $\pi'(0)$ has the claimed form, and it is easy to see that this representation is unique. Conversely, given $v, w \in \pi^\perp$, the curve of planes $\pi(t) := (a + v(t)) \wedge (b + w(t)) / \left|(a + v(t)) \wedge (b + w(t)) \right|$ satisfies $\pi'(0) = a \wedge w + v \wedge b$.
  \end{proof}
\end{prop}

The manifold $\Gr_2^+(\R^{n+1})$, with the Euclidean metric induced from $\extpow^2 \R^{n+1}$, is naturally a K\"ahler manifold, and in fact it embeds in $\CP^n$ as a standard complex quadric (see \cite{chern-grassmannian} or \cite{orriols-phd} for more details). With the description of tangent spaces from \cref{prop:tangent-grass}, its symplectic form is simply
\begin{equation}
  \label{eq:symplectic-form-gr}
  \omega(a \wedge w + v \wedge b, a \wedge w' + v' \wedge b)
    = \langle v, w' \rangle - \langle w, v' \rangle.
\end{equation}

\begin{prop}
  \label{prop:gauss-lagrangian}
  The generalized Gauss map
  \begin{align*}
    G : \S(N\Sigma) & \longrightarrow \Gr_2^+(\R^{n+1}) \\
      (x, \nu) & \longmapsto \phi(x) \wedge \nu
  \end{align*}
  of an immersion $\phi : \Sigma \to \S^n$ is Lagrangian with respect to the symplectic form \eqref{eq:symplectic-form-gr}.
  \begin{proof}
    Fix some point $(x, \nu) \in \S(N\Sigma)$ and let $V, V' \in T_x \Sigma$ and $\xi, \xi' \in T_\nu(\S(N_x\Sigma))$.
    We have, from the definition of the second fundamental form $A^\nu = \nu \cdot A(\cdot, \cdot)$,
    \[
      \Dd_{(V, \xi)} G = \Dd_V \phi \wedge \nu + \phi \wedge (\xi - A^\nu(\Dd_V \phi)),
    \]
    with $\Dd_V \phi$ and $A^\nu(\Dd_V \phi)$ tangent to $\phi(\Sigma)$ and $\xi$ orthogonal to both $\nu$ and $\phi(x)$, since $\xi \in N_x \Sigma \subset T_{\phi(x)}\S^n = \phi(x)^\perp$. With a similar expression for $(V', \xi')$, we have:
    \begin{align*}
      \pushQED{\qed}
      G^*\omega((V, \xi), (V', \xi'))
      &= \langle \Dd_V \phi, \xi' - A^\nu(\Dd_{V'} \phi)\rangle - \langle \Dd_{V'} \phi, \xi  - A^\nu(\Dd_{V} \phi) \rangle \\
      &= A^\nu(\Dd_{V} \phi, \Dd_{V'} \phi) - A^\nu(\Dd_{V'} \phi, \Dd_V \phi)
      = 0.
      \qedhere
    \end{align*}
  \end{proof}
\end{prop}

For the rest of this paper we will restrict to an ambient sphere of dimension $n = 3$ and stick to \cref{def:gauss-map}. In this case we have a more precise characterization of the Grassmannian. Recall that the Hodge star $\star : \extpow^2 \R^4 \to \extpow^2 \R^4$ is an involution, and denote by $\extpow^2_\pm$ its $(\pm 1)$-eigenspaces.

\begin{prop}
  \label{prop:grass-spheres}
  The map $I : \extpow^2 \R^4 \to \extpow^2_+ \R^4 \times \extpow^2_- \R^4$, 
  \begin{equation}
    \label{eq:isometry-I}
    I(\xi) := \left( \frac{\xi + \star \xi}{2}, \frac{\xi - \star \xi}{2} \right),
  \end{equation}
  is a linear isometry and identifies the Grassmannian $\Grass$ with the product $\S^2_+(1 / \sqrt{2}) \times \S^2_-(1 / \sqrt{2})$ of the spheres of radius $1 / \sqrt{2}$ in $\extpow^2_\pm \R^4$.
  \begin{proof}
    Clearly $I$ is just the decomposition into the orthogonal sum $\extpow^2 \R^4 = \extpow^2_+ \R^4 \oplus \extpow^2_- \R^4$, so it is an isometry. Recall that a $2$-vector $\xi$ is simple if and only if $\xi \wedge \xi = 0$. Then we have, writing $\xi = \xi_+ + \xi_-$,
    \[
      |\xi_+|^2
      = \frac{1}{4} \left( |\xi|^2 + |\star \xi|^2 + \langle \xi, \star \xi \rangle + \langle \star \xi, \xi \rangle\right)
      = \frac{1}{2} \left( |\xi|^2  + \star (\xi \wedge \xi) \rangle\right)
    \]
    and similarly
    \[
      |\xi_-|^2
      = \frac{1}{4} \left( |\xi|^2 + |\star \xi|^2 - \langle \xi, \star \xi \rangle - \langle \star \xi, \xi \rangle\right)
      = \frac{1}{2} \left( |\xi|^2  - \star (\xi \wedge \xi) \rangle\right).
    \]
    Therefore $\xi$ corresponds to a $2$-plane if and only if $|\xi_+| = |\xi_-|$, in which case both are equal to $1 / \sqrt{2}$.
  \end{proof}
\end{prop}

\subsection{Computations using quaternions}
It turns out that by identifying $\R^4$ with the space of quaternions $\Quat$, the isomorphism $I$ and many of the computations associated with it can be carried out more directly. Recall that the real basis $\{ \bm{1}, \bm{i}, \bm{j}, \bm{k} \}$ satisfies
\begin{align*}
  \star (\bm{1} \wedge \bm{i}) = \bm{j} \wedge \bm{k}, \qquad
  \star (\bm{1} \wedge \bm{j}) = \bm{k} \wedge \bm{i}, \qquad
  \star (\bm{1} \wedge \bm{k}) = \bm{i} \wedge \bm{j},
\end{align*}
in such a way that
\[
  e_{\bm{i}}^\pm := \frac{\bm{1} \wedge \bm{i} \pm \bm{j} \wedge \bm{k}}{\sqrt{2}}, \qquad
  e_{\bm{j}}^\pm := \frac{\bm{1} \wedge \bm{j} \pm \bm{k} \wedge \bm{i}}{\sqrt{2}}, \qquad
  e_{\bm{k}}^\pm := \frac{\bm{1} \wedge \bm{k} \pm \bm{i} \wedge \bm{j}}{\sqrt{2}}
\]
constitute orthonormal bases of $\bigwedge^2_\pm \R^4$; we will fix the orientation of $\extpow^2_\pm \R^4$ so that this basis is positive. Now let $\bm{a} = a_0 + a_1 \bm{i} + a_2 \bm{j} + a_3 \bm{k}$ and $\bm{b} = b_0 + b_1 \bm{i} + b_2 \bm{j} + b_3 \bm{k}$ be two unit quaternions orthogonal to each other. Then
\begin{align*}
  \bm{b} \overline{\bm{a}}
  & = (b_0 \bm{1} + b_1 \bm{i} + b_2 \bm{j} + b_3 \bm{k}) (a_0 \bm{1} - a_1 \bm{i} - a_2 \bm{j} - a_3 \bm{k})  \\
  &= (b_1 a_0 - b_0 a_1 + b_3 a_2 - b_2 a_3) \bm{i} \\
  & \quad + (b_2 a_0 - b_0 a_2 + b_1 a_3 - b_3 a_1) \bm{j} \\
  & \quad + (b_3 a_0 - b_0 a_3 + b_2 a_1 - b_1 a_2) \bm{k}
\end{align*}
and
\begin{align*}
  \overline{\bm{a}} \bm{b}
  & = (a_0 \bm{1} - a_1 \bm{i} - a_2 \bm{j} - a_3 \bm{k}) (b_0 \bm{1} + b_1 \bm{i} + b_2 \bm{j} + b_3 \bm{k}) \\
  &= (a_0 b_1 - a_1 b_0 + a_3 b_2 - a_2 b_3) \bm{i} \\
  & \quad + (a_0 b_2 - a_2 b_0 + a_1 b_3 - a_3 b_1) \bm{j} \\
  & \quad + (a_0 b_3 - a_3 b_0 + a_2 b_1 - a_1 b_2) \bm{k}.
\end{align*}
Notice that the real parts vanish by orthogonality. On the other hand we have
\begin{align*}
  \bm{a} \wedge \bm{b}
  & = (a_0 \bm{1} + a_1 \bm{i} + a_2 \bm{j} + a_3 \bm{k}) \wedge (b_0 \bm{1} + b_1 \bm{i} + b_2 \bm{j} + b_3 \bm{k}) \\
  & =
  (b_1 a_0 - b_0 a_1) \bm{1} \wedge \bm{i} 
  + (b_2 a_0 - b_0 a_2) \bm{1} \wedge \bm{j} 
  + (b_3 a_0 - b_0 a_3) \bm{1} \wedge \bm{k} \\
  &\quad + (b_2 a_1 - b_1 a_2) \bm{i} \wedge \bm{j}
  + (b_3 a_2 - b_2 a_3) \bm{j} \wedge \bm{k}
  + (b_1 a_3 - b_3 a_1) \bm{k} \wedge \bm{i} \\
  & =
    \frac{1}{\sqrt{2}} ((b_1 a_0 - b_0 a_1) + (b_3 a_2 - b_2 a_3)) e_{\bm{i}}^+
  + \frac{1}{\sqrt{2}} ((b_1 a_0 - b_0 a_1) - (b_3 a_2 - b_2 a_3)) e_{\bm{i}}^- \\
  &\quad
  + \frac{1}{\sqrt{2}} ((b_2 a_0 - b_0 a_2) + (b_1 a_3 - b_3 a_1)) e_{\bm{j}}^+
  + \frac{1}{\sqrt{2}} ((b_2 a_0 - b_0 a_2) - (b_1 a_3 - b_3 a_1)) e_{\bm{j}}^- \\
  &\quad
  + \frac{1}{\sqrt{2}} ((b_3 a_0 - b_0 a_3) + (b_2 a_1 - b_1 a_2)) e_{\bm{k}}^+
  + \frac{1}{\sqrt{2}} ((b_3 a_0 - b_0 a_3) - (b_2 a_1 - b_1 a_2)) e_{\bm{k}}^-.
\end{align*}
Comparing the expressions for $\bm{b} \overline{\bm{a}}$ and $\overline{\bm{a}} \bm{b}$ above with this, we see that under the identifications $\bm{i} \leftrightarrow e_{\bm{i}}^\pm$, $\bm{j} \leftrightarrow e_{\bm{j}}^\pm$, $\bm{k} \leftrightarrow e_{\bm{k}}^\pm$ which yield isometries between the space of purely imaginary quaternions $\R^3 \subset \Quat$ and the spaces $\extpow^2_\pm \R^4$, the isomorphism $I$ can be described as
\begin{equation}
  I(\bm{a} \wedge \bm{b}) = \frac{1}{\sqrt{2}} (\bm{b} \overline{\bm{a}}, \overline{\bm{a}} \bm{b}).
\end{equation}
In coordinates, we have
\[
  \bm{b} \overline{\bm{a}}
  = \langle \bm{i}, \bm{b} \overline{\bm{a}} \rangle \bm{i}
   + \langle \bm{j}, \bm{b} \overline{\bm{a}} \rangle \bm{j}
   + \langle \bm{k}, \bm{b} \overline{\bm{a}} \rangle \bm{k}
  = \langle \bm{i} \bm{a}, \bm{b} \rangle \bm{i}
   + \langle \bm{j} \bm{a}, \bm{b} \rangle \bm{j}
   + \langle \bm{k} \bm{a}, \bm{b} \rangle \bm{k}
\]
and
\[
  \overline{\bm{a}} \bm{b}
  = \langle \bm{i}, \overline{\bm{a}} \bm{b} \rangle \bm{i}
   + \langle \bm{j}, \overline{\bm{a}} \bm{b} \rangle \bm{j}
   + \langle \bm{k}, \overline{\bm{a}} \bm{b} \rangle \bm{k}
   = \langle \bm{a} \bm{i}, \bm{b} \rangle \bm{i}
   + \langle \bm{a} \bm{j}, \bm{b} \rangle \bm{j}
   + \langle \bm{a} \bm{k}, \bm{b} \rangle \bm{k},
\]
which corresponds to \eqref{eq:quaternion-coordinates}.

It will be useful later to characterize the sets of planes that contain a given vector $\bm{a} \in \Quat$.  For that, recall that every unit quaternion $\bm{a}$ induces a rotation $R_{\bm{a}}$ that acts on the space $\R^3 \subset \Quat$ of purely imaginary quaternions as
\begin{equation}
  R_{\bm{a}} (\bm{b}) = \bm{a}^{-1} \bm{ba} = \overline{\bm{a}} \bm{ba}.
  \label{eq:rotation-quaternions}
\end{equation}

\begin{prop}
  \label{prop:graph-rotation-plane}
  Let $\bm{a}$ be a unit quaternion. Then the image by $I$ of the set of $2$-planes in $\Grass$ containing $\bm{a}$ is the graph of the rotation $R_{\bm{a}}$ acting on $\S^2\left(1 / \sqrt{2}\right)$.
  \begin{proof}
    This is just a computation:
    \begin{align*}
      I(\{ \bm{a} \wedge \bm{b} \,:\, |\bm{b}| = 1, \bm{b} \perp \bm{a} \})
      &= \left\{ \frac{1}{\sqrt{2}} (\bm{b} \overline{\bm{a}}, \overline{\bm{a}} \bm{b}) \, : \, |\bm{b}| = 1, \bm{b} \perp \bm{a} \right\} \\
      &= \left\{ \frac{1}{\sqrt{2}} (\bm{c}, \overline{\bm{a}} \bm{ca}) \, : \, |\bm{c}| = 1, \Re \bm{c} = 0 \right\} \\
      &= \left\{ \frac{1}{\sqrt{2}} (\bm{c}, R_{\bm{a}}(\bm{c})) \, : \, \bm{c} \in \S^2 \subset \R^3 \subset \Quat \right\}.
    \end{align*}
    Here we have used that $\Re \bm{b} \overline{\bm{a}} = 0$ if and only if $\bm{b}$ is orthogonal to $\bm{a}$.
  \end{proof}
\end{prop}
Next we study the symplectic structure in this representation. Define the following symplectic forms on $\S^2_\pm(1/\sqrt{2})$: for a point $\bm{p} \in \S^2_\pm$ we set
\[
  \omega_\pm|_{\bm{p} / \sqrt{2}}(\bm{x}, \bm{y}) := \langle \bm{x}, \bm{y} \bm{p} \rangle.
\]
Clearly $\left( \S^2_\pm(1/\sqrt{2}), \omega_\pm \right)$ are both symplectomorphic to a standard $\S^2$ of total symplectic area $4 \pi (1 / \sqrt{2})^2 = 2\pi$ (with the orientation induced by $\omega_\pm$).

\begin{prop}
  The map $I$ is a symplectomorphism between $(\Grass, \omega)$ and $(\S^2_+, \omega_+) \times (\S^2_-, -\omega_-)$.
  \begin{proof}
    Notice that
    \[
      \Dd I |_{\bm{a} \wedge \bm{b}}(\bm{a} \wedge \bm{w} + \bm{v} \wedge \bm{b})
      = \frac{1}{\sqrt{2}} (\bm{b} \overline{\bm{v}} + \bm{w} \overline{\bm{a}}, \overline{\bm{a}} \bm{w} + \overline{\bm{v}} \bm{b}).
    \]
    Therefore
    \begin{align*}
      I^*(\omega_+ - \omega_-)|_{\bm{a} \wedge \bm{b}}((\bm{v}, \bm{w}), (\bm{v}', \bm{w}'))
      = \frac{1}{2} \bigl( \langle \bm{b} \overline{\bm{v}} &+ \bm{w} \overline{\bm{a}}, (\bm{b} \overline{\bm{v}'} + \bm{w}' \overline{\bm{a}}) \bm{b} \overline{\bm{a}} \rangle \\
      &- \langle \overline{\bm{a}} \bm{w} + \overline{\bm{v}} \bm{b}, (\overline{\bm{a}} \bm{w}' + \overline{\bm{v}'} \bm{b}) \overline{\bm{a}} \bm{b} \rangle \bigr).
    \end{align*}
    We compute term by term. First look at the term with $\bm{v}, \bm{v}'$:
    \[
      \langle \bm{b} \overline{\bm{v}} , \bm{b} \overline{\bm{v}'} \bm{b} \overline{\bm{a}} \rangle
      - \langle \overline{\bm{v}} \bm{b}, \overline{\bm{v}'} \bm{b} \overline{\bm{a}} \bm{b} \rangle
      = \langle \overline{\bm{v}}, \overline{\bm{v}'} \bm{b} \overline{\bm{a}} \rangle
      - \langle \overline{\bm{v}}, \overline{\bm{v}'} \bm{b} \overline{\bm{a}} \rangle
      = 0.
    \]
    Similarly the term with $\bm{w}, \bm{w}'$ is
    \[
      \langle  \bm{w} \overline{\bm{a}}, \bm{w}' \overline{\bm{a}} \bm{b} \overline{\bm{a}} \rangle
      - \langle \overline{\bm{a}} \bm{w}, \overline{\bm{a}} \bm{w}' \overline{\bm{a}} \bm{b} \rangle
      = \langle  \bm{w}, \bm{w}' \overline{\bm{a}} \bm{b} \rangle
      - \langle \bm{w}, \bm{w}' \overline{\bm{a}} \bm{b} \rangle
      = 0.
    \]
    The term with $\bm{v}, \bm{w}'$ is
    \begin{align*}
      \langle \bm{b} \overline{\bm{v}} , \bm{w}' \overline{\bm{a}} \bm{b} \overline{\bm{a}} \rangle
      - \langle \overline{\bm{v}} \bm{b}, \overline{\bm{a}} \bm{w}' \overline{\bm{a}} \bm{b} \rangle
      &=\langle \bm{b} \overline{\bm{v}} , \bm{w}' (-\overline{\bm{b}} \bm{a}) \overline{\bm{a}} \rangle
      - \langle \overline{\bm{v}}, \overline{\bm{a}} \bm{w}' \overline{\bm{a}} \rangle
      = -\langle \bm{b} \overline{\bm{v}} , \bm{w}' \overline{\bm{b}} \rangle
      - \langle \bm{v}, \bm{a} \overline{\bm{w}'} \bm{a} \rangle \\
      &= \langle \bm{v} \overline{\bm{b}} , \bm{w}' \overline{\bm{b}} \rangle
      + \langle \bm{v}, \bm{w}' \overline{\bm{a}} \bm{a} \rangle
      = 2 \langle \bm{v} , \bm{w}' \rangle.
    \end{align*}
    using that $\bm{a} \perp \bm{b}$, $\bm{b} \perp \bm{v}$ and $\bm{w}' \perp \bm{a}$. Finally, the last term with $\bm{w}, \bm{v}'$ similarly gives
    \begin{align*}
      \langle \bm{w} \overline{\bm{a}}, \bm{b} \overline{\bm{v}'} \bm{b} \overline{\bm{a}} \rangle
      - \langle \overline{\bm{a}} \bm{w}, \overline{\bm{v}'} \bm{b} \overline{\bm{a}} \bm{b} \rangle
      &= \langle \bm{w}, \bm{b} \overline{\bm{v}'} \bm{b} \rangle
      - \langle \bm{w}, \bm{a} \overline{\bm{v}'} \bm{b} \overline{\bm{a}} \bm{b} \rangle \\
      &= \langle \bm{w}, (-\bm{v}' \overline{\bm{b}}) \bm{b} \rangle
      - \langle \bm{w}, (-\bm{v}' \overline{\bm{a}}) \bm{b} (-\overline{\bm{b}} \bm{a}) \rangle \\
      &= -\langle \bm{w}, \bm{v}' \rangle
      - \langle \bm{w}, \bm{v}' \overline{\bm{a}} \bm{a} \rangle
      = -2\langle \bm{w}, \bm{v}' \rangle.
    \end{align*}
    Putting all terms together gives the equality.
  \end{proof}
\end{prop}

Finally we compute the degree of the map $G : \Sigma \to \Grass$.
\begin{prop}
  \label{prop:degree-gauss-map}
  If $\Sigma$ has genus $g$, then for any immersion,
  \begin{equation}
    \label{eq:degree-gauss-map}
    (I \circ G)_\# [\Sigma] = (1 - g, 1 - g) \in \Z^2 \cong H_2(\S^2_+ \times \S^2_-) \cong H_2(\Grass),
  \end{equation}
  where we choose generators of $\S^2_\pm$ so that $\omega_\pm$ pairs with them positively.
  \begin{proof}
    Clearly $G_\# [\Sigma] = (k, k)$ for some integer $k$, since $G$ is Lagrangian:
    \[
      (I \circ G)^*(\omega_+ - \omega_-)
      = G^*\omega = 0.
    \]
    We compute first $I^*\omega_+$ on two tangent vectors $(v, w), (v', w') \in T_{a \wedge b}\Grass$:
    \begin{align*}
      I^*\omega_+ |_{\bm{a} \wedge \bm{b}} ((\bm{v}, \bm{w}), (\bm{v}', \bm{w}'))
      &= \frac{1}{2} \langle \bm{b} \overline{\bm{v}} + \bm{w} \overline{\bm{a}},
      \bm{b} \overline{\bm{v}'} \bm{b} \overline{\bm{a}} + \bm{w}' \overline{\bm{a}} \bm{b} \overline{\bm{a}} \rangle \\
      &= \frac{1}{2} \langle -\bm{v} \overline{\bm{b}} + \bm{w} \overline{\bm{a}},
      -\bm{v}' \overline{\bm{a}} - \bm{w}' \overline{\bm{b}} \rangle \\
      &= \frac{1}{2} \left(
        \langle \bm{v} \overline{\bm{b}},
        \bm{v}' \overline{\bm{a}} \rangle
        + \langle \bm{v} \overline{\bm{b}},
         \bm{w}' \overline{\bm{b}} \rangle
        - \langle \bm{w} \overline{\bm{a}},
        \bm{v}' \overline{\bm{a}} \rangle
        - \langle \bm{w} \overline{\bm{a}},
        \bm{w}' \overline{\bm{b}} \rangle
  \right) \\
      &= \frac{1}{2} \left(
        \langle \bm{v} \overline{\bm{b}}, \bm{v}' \overline{\bm{a}} \rangle
        + \langle \bm{v}, \bm{w}' \rangle
        - \langle \bm{w}, \bm{v}' \rangle
        - \langle \bm{w} \overline{\bm{a}}, \bm{w}' \overline{\bm{b}} \rangle \right).
    \end{align*}
    Recall that $\partial_{x_i} G = (\partial_{x_i} \phi, \partial_{x_i} \nu) = (\partial_{x_i} \phi, -\kappa_i \partial_{x_i} \phi)$ in coordinates adapted to the principal directions at a point. In particular, choosing $a = \phi, b = \nu, v = \partial_{x_1} \phi, v' = \partial_{x_2} \phi, w = -\kappa_1 v$ and $w' = -\kappa_2 v'$, we have that
    Thus we get
    \begin{align*}
      (I \circ G)^*\omega_+ |_{\bm{\phi} \wedge \bm{\nu}} (\partial_{x_1}, \partial_{x_2})
      &= \frac{1}{2} \left(
        \langle \bm{v} \overline{\bm{b}}, \bm{v}' \overline{\bm{a}} \rangle
        - \kappa_1 \kappa_2 \langle \bm{v} \overline{\bm{a}}, \bm{v}' \overline{\bm{b}} \rangle \right) \\
      &= \frac{1}{2} \left(
        \langle \bm{v}, \bm{v}' \overline{\bm{a}} \bm{b} \rangle
      - \kappa_1 \kappa_2 \langle \bm{v}, \bm{v}' \overline{\bm{b}} \bm{a} \rangle \right) \\
      &= \frac{1}{2} (1 + \kappa_1 \kappa_2) \langle \bm{v}, \bm{v}' \overline{\bm{a}} \bm{b} \rangle
      = \frac{1}{2} (1 + \kappa_1 \kappa_2) \Re (\overline{\bm{v}} \bm{v}' \overline{\bm{a}} \bm{b}).
    \end{align*}
    Now observe that the expression $\Re (\overline{\bm{z_1}} \bm{z_2} \overline{\bm{z_3}} \bm{z_4})$ always gives the same value whenever $z_1, z_2, z_3, z_4$ are an oriented orthonormal basis of $\R^4$ (one can check this by observing that it remains unchanged by rotations between two elements of the basis). Since $\{ \partial_{x_1} \phi / |\partial_{x_1} \phi|, \partial_{x_2} \phi / |\partial_{x_2} \phi|, \phi, \nu \}$ is such a basis, and so is $\{\bm{1}, \bm{i}, \bm{j}, \bm{k}\}$, we have that
    \[
      (I \circ G)^*\omega_+ |_{\bm{\phi} \wedge \bm{\nu}}
      = \frac{1}{2} (1 + \kappa_1 \kappa_2) |\partial_{x_1} \phi \wedge \partial_{x_2} \phi| \Re (\overline{\bm{1}} \bm{i} \overline{\bm{j}} \bm{k})  \dd x_1 \wedge \dd x_2
      = \frac{1}{2} K \vol_\phi.
    \]
    Finally, by Gauss--Bonnet,
    \[
      \pushQED{\qed}
      k = \deg (I^+ \circ G)
      = \frac{\int_\Sigma (I \circ G)^*\omega_+}{\int_{\S^2_+} \omega_+}
      = \frac{1}{2} \frac{\int_\Sigma K \vol_\phi}{2\pi}
      = \frac{1}{2} \frac{2\pi(2 - 2g)}{2\pi}
      = 1 - g.
      \qedhere
    \]
  \end{proof}
\end{prop}

\section{Area of Gauss maps and total absolute curvature}
\label{sec:area-gauss-tac}

In this section we study the area of the Gauss map of an immersion and prove \cref{thm:area-gauss-bound}. We start with a computation:

\begin{lemma}
  Let $\phi : \Sigma \to \S^3$ be an immersion and $G$ its Gauss map. Then its (parametric) area is given by
  \begin{equation}
    \AG(\phi)
    := \Area(G)
    = \int_\Sigma \sqrt{4 H^2 + (K - 2)^2} \dvol_\Sigma
    = \int_\Sigma \sqrt{1+\kappa_1^2} \sqrt{1+\kappa_2^2} \dvol_\Sigma.
    \label{area-gauss}
  \end{equation}
  Here $\kappa_1$ and $\kappa_2$ denote the principal curvatures of $\Sigma$, $H = \tfrac{\kappa_1 + \kappa_2}{2}$ is the mean curvature, and $K = 1 + \kappa_1 \kappa_2$ is the Gauss curvature.
  \begin{proof}
    It is enough to compute the Jacobian of $G$ at a point choosing coordinates $x_1, x_2$ that diagonalize the second fundamental form at that point. Using the description of $T_{\phi \wedge \nu}\Grass$ given by \cref{prop:tangent-grass},
    \[
      \partial_{x_i} G
      = \partial_{x_i} \phi \wedge \nu + \phi \wedge \partial_{x_i} \nu
      = \partial_{x_i} \phi \wedge \nu - \kappa_i \phi \wedge \partial_{x_i} \phi,
    \]
    we see that $\partial_{x_1} G$ is orthogonal to $\partial_{x_2} G$, so
    \[
      \Jac(G)
      = |\partial_{x_1} G| \cdot |\partial_{x_2} G|
      = \sqrt{1 + \kappa_1^2} \sqrt{1 + \kappa_2^2} \, |\partial_{x_1} \phi| \cdot |\partial_{x_2} \phi|
      = \sqrt{1 + \kappa_1^2} \sqrt{1 + \kappa_2^2} \, \Jac(\phi).
    \]
    The alternative expression follows from a computation:
    \begin{align*}
      \pushQED{\qed}
      (1 + \kappa_1^2) (1 + \kappa_2^2)
      &= \kappa_1^2 + \kappa_2^2 + 2 \kappa_1 \kappa_2 + \kappa_1^2 \kappa_2^2 - 2 \kappa_1 \kappa_2 + 1
      = (\kappa_1 + \kappa_2)^2 + (\kappa_1 \kappa_2 - 1)^2 \\
      &= 4 H^2 + (K - 2)^2.
      \qedhere
    \end{align*}
  \end{proof}
\end{lemma}

We will relate this quantity with another geometric functional called the total absolute curvature of the immersion. This functional was introduced by Chern and Lashof \cite{chern-lashof1}, where they proved a lower bound in arbitrary dimensions independent of the topology of the submanifold.

In order to define it, consider the map $F : U\Sigma \to \S^3$, where $U\Sigma$ denotes the bundle of unit normal vectors to $\Sigma$ in $\R^4$, which assigns to each pair $(x,q) \in U\Sigma$ the same vector $q$ viewed as a vector at the origin, i.e.~as an element of $\S^3$. Since we are assuming that $\Sigma$ has a global unit normal vector $\nu$, the normal bundle of $\Sigma$ in $\R^4$ will be trivialized by the sections $\phi$ and $\nu$, and therefore from now on we will identify $U\Sigma = \Sigma \times \S^1$ with $(x, \theta) \in \Sigma \times \S^1$ corresponding to $\phi(x) \cos \theta + \nu(x) \sin \theta$.
\begin{defn}
  \label{def:total-absolute-curvature}
  The total volume swept out by the map $F$ is called the \emph{total absolute curvature} of the immersion:
  \begin{equation}
    \label{eq:total-absolute-curvature}
    \TAC(\phi) = \int_{U \Sigma} \Jac(F) \dvol_{U\Sigma} = \int_\Sigma \int_{\S^1} \Jac(F)(x, \theta) \dd \theta \dvol_{\Sigma}(x).
  \end{equation}
\end{defn}

More precisely, the integral $\int_{\S^1} \Jac(F)(x, \theta) \, \dd \theta$ is called the absolute curvature (or Lipschitz--Killing curvature) of $\phi$ at $x$. In a subsequent paper, Chern and Lashof \cite{chern-lashof2} proved the following sharp lower bound for every genus $g \geq 0$:
\begin{thm}[Chern--Lashof]
  \label{thm:chern-lashof}
  Let $\phi : \Sigma \to \S^3$ be a smooth immersion of a closed oriented surface $\Sigma$ of genus $g$. Then $\TAC(\phi) \geq 2\pi^2(2g + 2)$.
\end{thm}

The constant in \cref{thm:chern-lashof} is optimal for every $g$, but it is only attained for $g = 0$ and $1$. Surfaces for which equality holds are called \emph{taut immersions}; the taut spheres are precisely all round (not necessarily great) spheres in $\S^3$ \cite{chern-lashof1}, whereas for higher genus, Banchoff \cite{banchoff} proved that there are no taut surfaces of genus bigger than one, and that every taut torus is a Dupin cyclide, namely the image by a M\"obius transformation of a product torus $\S^1(\rho) \times \S^1(\sqrt{1-\rho^2}) \subset \S^3$. See the book \cite{cecil-ryan} or the surveys in the collection \cite{tight-taut-msri} for an introduction and a comprehensive list of references on the topic.

We will discuss the optimality of the inequality (even though it is always strict) for genus $g > 1$ in the next section; for now let us show how \cref{thm:area-gauss-bound} follows from it.

\begin{proof}[Proof of \cref{thm:area-gauss-bound}]
  Let us compute the Jacobian $\Jac(F)$ at an arbitrary point $x \in \Sigma$ by choosing local coordinates such that $\del_{x_i} \nu(x) = -\kappa_i(x) \del_{x_i} \phi(x)$. The partial derivatives of $F$ are
  \[
    \partial_{x_i} F(x, \theta) = (\cos \theta - \kappa_i \sin \theta) \partial_{x_i} \phi 
    \qquad \text{and} \qquad
    \partial_{\theta} F(x, \theta) = \phi(x) \sin \theta - \nu(x) \cos \theta,
  \]
  so that
  \[
    \partial_{\theta} F \wedge F
    = (\phi \sin \theta - \nu \cos \theta) \wedge (\phi \cos \theta + \nu \sin \theta)
    = \phi \wedge \nu
  \]
  and thus
  \begin{align*}
    \Jac(F)(x, \theta)
    &= |\partial_{x_1} F \wedge \partial_{x_2} F \wedge \partial_\theta F \wedge F| \\
    &= |\cos \theta - \kappa_1 \sin \theta| |\cos \theta - \kappa_2 \sin \theta| |\partial_{x_1} \phi \wedge \partial_{x_2} \phi \wedge \phi \wedge \nu| \\
    &= |\cos \theta - \kappa_1 \sin \theta| |\cos \theta - \kappa_2 \sin \theta| |\partial_{x_1} \phi \wedge \partial_{x_2}| \\
    &= |\cos \theta - \kappa_1 \sin \theta| |\cos \theta - \kappa_2 \sin \theta| \Jac(\phi).
  \end{align*}
  Let $\tilde \theta_i(x) \in \left( 0, \pi \right)$ be such that $\cot \tilde \theta_i(x) = \kappa_i(x)$ and hence
  \[
    \cos \tilde \theta_i = \frac{\kappa_i}{\sqrt{1 + \kappa_i^2}}
    \qquad \text{and} \qquad
    \sin \tilde \theta_i = \frac{1}{\sqrt{1 + \kappa_i^2}}.
  \]
  Then
  \[
    \cos \theta - \kappa_i \sin \theta
    = \sqrt{1+\kappa_i^2} \left( \sin \tilde \theta_i \cos \theta - \cos \tilde \theta_i \sin \theta \right)
    = \sqrt{1+\kappa_i^2} \sin \left( \tilde \theta_i -  \theta \right)
  \]
  and we can write the total absolute curvature as
  \[
    \TAC(\phi) = \int_\Sigma \sqrt{1 + \kappa_1(x)^2} \sqrt{1 + \kappa_2(x)^2} \int_{\S^1} |\sin(\tilde \theta_1(x) - \theta) \sin(\tilde \theta_2(x) - \theta)| \dd \theta \dvol_{\Sigma}(x).
  \]
  The inner integral depends only on $|\tilde \theta_1(x) - \tilde \theta_2(x)|$ and can be computed explicitly, although Cauchy--Schwarz suffices to obtain an upper bound:
  \[
    \int_{\S^1} |\sin(\tilde \theta_1 - \theta) \sin(\tilde \theta_2 - \theta)| \dd \theta
    \leq \left( \int_{\S^1} \sin(\tilde \theta_1 - \theta)^2 \dd \theta \right)^{1 / 2}
    \left( \int_{\S^1} \sin(\tilde \theta_2 - \theta)^2 \dd \theta \right)^{1 / 2} = \pi.
  \]
  From this the inequality follows. Equality can only hold if $\tilde \theta_1(x) = \tilde \theta_2(x)$ for every $x \in \Sigma$, which implies that $\phi$ is umbilic everywhere and then it is a classical fact that $\phi$ covers a round sphere. Since $\phi$ is an immersion, it follows that $g = 0$, and the last claim in the theorem is a consequence of the description of Gauss maps of round spheres from \eqref{eq:gauss-map-round-sphere}.
\end{proof}

\section{Optimality and existence of almost minimizers}
\label{sec:min-sequences}

We have seen that for an immersion $\phi : \Sigma \to \S^3$ of a surface of genus $g$,
\begin{equation}
  \label{eq:lower-bound-area-gauss}
  \AG(\phi) \geq 4\pi (1 + g)
\end{equation}
and equality can only hold when $g = 0$, in which case $\phi$ is umbilic and hence embeds $\Sigma$ as a round sphere.

Let us first show that this inequality is sharp, that is, that the constant in the right hand side of \eqref{eq:lower-bound-area-gauss} cannot be improved.

\begin{prop}
  \label{prop:concentrating-sequence}
  For any $g \geq 0$ and any $\eps > 0$ there exists an immersion $\phi : \Sigma \to \S^3$ of a surface $\Sigma$ of genus $g$ such that $\AG(\phi) \leq 4\pi (1 + g) + \eps$.
  \begin{proof}
    The result follows from the corresponding result in $\R^3$. Namely, it is well known that for immersions $\phi_0 : \Sigma \to \R^3$ with unit normal $\nu_0$ of surfaces $\Sigma$ of arbitrary genus $g$, the analogous inequality
    \[
      \Area(\nu_0) \geq 4\pi (g + 1)
    \]
    holds and has plenty of minimizers, so-called tight surfaces. For $g = 0$ they agree precisely with the boundaries of convex bodies \cite{chern-lashof1}. For $g > 0$, one can obtain a tight surface of genus $g$ by starting with a smooth convex body containing two parallel planar regions and carving $g$ holes connecting them, each with nonpositive Gauss curvature everywhere (see \cite{cecil-ryan-top-cycles} for a complete characterization).

    Now given such an immersion $\phi_0 : \Sigma \to \R^3$ and $r > 0$, consider the immersion $\phi_r : \Sigma \to \S^3$ given by $\phi_r := \Pi^{-1} \circ \phi_0$, where $\Pi^{-1} : \R^3 \to \S^3$ is the inverse stereographic projection, and its unit normal vector $\nu_r : \Sigma \to \S^3$. It is immediate to check that
    \[
      \phi_r \longrightarrow p_S
      \qquad \text{and} \qquad
      \nu_r \longrightarrow \nu_0
      \qquad \text{as } r \to 0
    \]
    in $C^\infty(\Sigma)$. Therefore $\AG(\phi_r) = \Area(\phi_r \wedge \nu_r) \to \Area(p_S \wedge \nu_0) = \Area(\nu_0)$ and the same lower bound of $4\pi(g+1)$ is attained.
  \end{proof}
\end{prop}

In light of these examples and of the non-existence of minimizers, one may be tempted to think that the concentrating behavior described in \cref{prop:concentrating-sequence} is the only possible way to approach the infimum for genus $g > 0$. The rest of this section will be dedicated to constructing more elaborate sequences in which there is a macroscopic round $\S^2$ (which we call $S$) which passes to the limit, and concentration happens along $g$ points in $S$. This is the content of \cref{thm:concentration-handles}.

The following family of surfaces (handles) will be the basic block of our construction.

\begin{lemma}
  \label{lem:low-curvature-handle}
There exists a family of smooth embedded surfaces $\{H_h\} \subset \R^3$ with the following properties:
\begin{itemize}
  \item $H_h$ has one boundary component and a neighborhood of the boundary is contained in $\R^2 \subset \R^3$.
  \item $H_h$ has genus $1$ and thus is homeomorphic to a torus with a disk removed.
  \item The total curvature of $H_h$ converges to $4 \pi$ as $h \searrow 0$.
\end{itemize}

\begin{proof}
We start with a smooth decreasing function $f : \R^+ \to \R$ with the following properties:
\begin{itemize}
  \item $f(r) = 1$ for $0 < r < 1$.
  \item $f(r) = \frac{7-2r}{4}$ for $2 < r < 3$.
  \item $f(r) = 0$ for $r > 4$.
  \item $f$ is concave on $(1, 2)$.
  \item $f$ is convex on $(3, 4)$.
\end{itemize}

Then define the function $g_h : \R^2 \to \R$, which is of class $\Cont^{1,1}$ (and can be made $\Cont^\infty$ with a more complicated definition or by smoothing out):
\begin{equation}
  g_h(x, y) :=
  \begin{cases}
    h f(\sqrt{x^2 + (y - 1)^2}), &y \geq 1 \\
    h f(|x|), &-1 \leq y \leq 1 \\
    h f(\sqrt{x^2 + (y + 1)^2}), &y \leq -1 \\
  \end{cases}
\end{equation}
and let $G_h \subset \R^3$ be the graph of $g_h$ over $\R^2$. It is easy to check that the total curvature of $G_h$ tends to $0$ as $h \to 0$ (for example by examining the image of the normal map).

Observe that the graphs of $g_h$ over $(-3,-2)\times(-1,1)$ and over $(2,3)\times(-1,1)$ are specular reflections of each other with respect to the plane $\{ x = 0 \}$. We finish the construction by inserting the \textit{cylinder} $C_h$ provided by \cref{lem:embedded-cylinder} below. This annulus attaches smoothly to these planar regions of $G_h$ and contributes exactly $4 \pi$ to the total absolute curvature, thanks to Gauss--Bonnet:
\[
  \int_{C_h} |K|
  = - \int_{C_h} K
  = - 2 \pi \chi(C_h) + \int_{\Gamma_1} k_g + \int_{\Gamma_2} k_g
  = 2 \pi + 2 \pi,
\]
where $\Gamma_1$ and $\Gamma_2$ denote the boundary components of $C_h$, which are smooth simple closed curves contained in planes.
\end{proof}
\end{lemma}

\begin{lemma}
  \label{lem:embedded-cylinder}
  Let $B$ be a ball in a plane $\pi \subset \R^3$ and consider the reflection $B'$ of $B$ with respect to a plane $\pi_0$ which is disjoint from $\pi$. Then there exists a smooth embedded surface $C \subset \R^3$ with the following properties:
  \begin{itemize}
    \item $C$ is homeomorphic to an annulus.
    \item The boundary components of $C$ are contained in $B$ and $B'$, respectively, and a neighborhood of each boundary component in $C$ is contained in the respective plane.
    \item $C$ has nonpositive Gauss curvature everywhere.
  \end{itemize}
  \begin{proof}
    We construct first a surface $C_0$ which is homeomorphic to an annulus, is contained in $B$ near one of its boundary components, and coincides with a cylinder meeting $\pi_0$ orthogonally near the other boundary component. Once we have this, $C_0$ and its reflection through $\pi_0$ can be joined smoothly and clearly define a surface $C$ with the desired properties.

    In the case when $\pi$ and $\pi_0$ are parallel, say $\pi_0$ is the plane $\{x=0\}$, a thin enough surface of revolution does the job. More precisely, suppose $\pi = \{x = -a\}$ and $B$ is the ball of radius $r$ in $\pi$ centered at $(-a, 0, 0)$. Choose a convex curve $\gamma = (\gamma_1, \gamma_2) : [-1, 0] \to \R^2_{x,z}$ contained in $\{ z > 0\}$ such that for $t$ near $-1$, $\gamma_1(t) \equiv -a$ and $\gamma_2'(t) < 0$, and for $t$ near $0$, $\gamma_2(t) \equiv \const > 0$ and $\gamma_1'(t) > 0$. Then if we let $C_0$ be the surface of revolution of $\gamma$ around the $X$-axis all properties are satisfied.

    In the general case it is enough to reduce to the previous case by an affine transformation that sends a plane $\tilde{\pi}$, which is parallel to $\pi_0$, to $\pi$. More precisely, if $\pi$ is given by $\{z = \lambda (x + a)\}$, let $A(x, y, z) = (x + \lambda^{-1} z, y, z)$. Then $A(\tilde{\pi}) = \pi$, where $\tilde{\pi} = \{ x = -a \}$ as before. Now we construct $\tilde{C_0}$ as above and define $C_0 := A (\tilde{C}_0)$, so that $C_0$ attaches smoothly to $\pi$. The cylindrical end of $\tilde{C}_0$ is preserved, since the $y$ and $z$ coordinates are not changed by the transformation $A$ and therefore the slices $A(\tilde C_0) \cap \{x = t\}$ are still circles for large $t$. Thus, possibly after shrinking $\tilde C_0$, $C_0$ meets $\pi_0$ orthogonally.

    The only thing to check is the nonnegativity of the Gauss curvature. But orientation-preserving affine maps never change the sign of the Gauss curvature, because the Gauss map $\nu$ of $C_0$ is related to the Gauss map $\tilde{\nu}$ of $\tilde{C_0}$ by $\nu = \alpha \circ \tilde{\nu}$, where $\alpha$ is an orientation-preserving diffeomorphism of the sphere induced by $A$, and thus $\nu^*\vol_{\S^2} = \tilde{\nu}^* (\alpha^* \vol_{\S^2})$ has the same sign as $\tilde{\nu}^* \vol_{\S^2}$.
  \end{proof}
\end{lemma}

Using the previous lemma we can prove \cref{thm:concentration-handles}, that is, we can exhibit sequences of smooth embeddings $\Sigma_j \subset \S^3$ of genus $g$ surfaces for which the measures associated to the area of the Gauss maps concentrate at any $g$ given ``bubble points'' lying on a round sphere $S$:
\begin{proof}[Proof of \cref{thm:concentration-handles}]
  Let $\Pi_k$ denote the stereographic projection from $-p_k$, so that $\Pi_k(p_k) = 0 \in \R^3$, and choose coordinates in such a way that $T_0(\Pi_k(S)) = \R^2_{x,y}$. Pick an arbitrary sequence $h_j \searrow 0$ and denote by $\eta_{h_j}$ the dilation $\eta_{h_j}(x) = h_j^{-1} x$. Then $\eta_{h_j} \circ \Pi_k(S)$ converges in $\Cont^\infty_\text{loc}$ to the plane $\R^2$ as $j \to \infty$. Therefore we can glue in smoothly a handle $H_{h_j}$ (as constructed in \cref{lem:low-curvature-handle}) around each $p_k$, and arguing as in \cref{prop:concentrating-sequence}, the total Gauss map area increases only by $4 \pi + o_j(1)$ while the genus increases by one with each handle.
  
  Naturally we can perform these gluings in disjoint regions for distinct $p_k$ and in this way obtain an embedded surface $\Sigma_{j}$. Now \eqref{eq:concentration-gauss-measure} is immediate, and the last claim follows from \cref{prop:graph-rotation-plane}.
\end{proof}

In particular, the Gauss curvatures of minimizing sequences may concentrate at different points and the genus is lost in the limit. Notice that the same construction works if we allow the points $p_j$ to be repeated, by just placing two copies of $H_h$ at the same point. Also notice the big flexibility at our disposal: the shape, position and scale of the handles can be chosen arbitrarily; this indicates that the problem is very degenerate and that convergence only in a very weak sense (as in \eqref{eq:concentration-gauss-measure} and \eqref{eq:convergence-currents-example}) is expected.

\section{Compactness for minimizing sequences}
\label{sec:compactness-gauss}

In this section we prove the main theorem of this work, \cref{thm:gauss-main}, which characterizes the behavior of almost minimizers of $\AG$ for a given genus. Notice that this almost solves Open Problem IV.1 from \cite{riviere-minmax-hierarchies}, except for the technical condition that the surfaces are embedded. Let us state here a stronger version of the theorem for the Legendrian lifts $\Lambda_j$ of the Gauss maps $G_j$, from which \cref{thm:gauss-main} will follow easily. The reason for that is that the image of these maps is related to the image of the embeddings, and this allows us to obtain more refined information about the limit.

Recall that $V_2(\R^4) = \{ (a, b) \in \S^3 \times \S^3 : a \cdot b = 0 \}$ is a contact manifold, with contact form $a \cdot \dd b$, and the map $\pi : V_2(\R^4) \to \Grass$ sending $(a, b)$ to $a \wedge b$ induces an isometry on all horizontal planes. Therefore the mapping area of $\Lambda$ and of $G = \pi \circ \Lambda$ agree, and their associated cycles are related by $G_\# \DBrack{\Sigma} = \pi_\# (\Lambda_\# \DBrack{\Sigma})$. We will refer to the current $\Lambda_\# \DBrack{\Sigma}$ as the \emph{normal cycle}, which is a fundamental object appearing in the theory of curvature measures (see for example \cite{morvan-generalized-curvatures} for a nice introduction).

\begin{thm}
  \label{thm:gauss-main-legendrian}
  Let $\phi_j : \Sigma \to \S^3$ be a sequence of smooth embeddings of a closed oriented surface $\Sigma$ of genus $g$, with normal maps $\nu_j : \Sigma \to \S^3$ and Legendrian Gauss maps $\Lambda_j = (\phi_j, \nu_j) : \Sigma \to V_2(\R^4)$. Suppose that $\AG(\phi_j) = \Area(\Lambda_j) \to 4\pi(1+g)$, which is the minimum in this class.

  Then there exists a decomposition $(\Lambda_j)_\# \DBrack{\Sigma} = N_j^+ - N_j^-$ into two integral cycles, a possibly degenerate round sphere $S \subset \S^3$, and $g$ points $p_1, \ldots, p_g \in S$ such that, after passing to a subsequence,
  \begin{equation}
    \label{eq:convergence-legendrian-cycles}
    N_j^+ \weakto (\Lambda_S)_\# \DBrack{S}
    \quad \text{and} \quad
     N_j^- \weakto \sum_{k=1}^g \DBrack{\{ p_k \} \times (\S^3 \cap p_k^\perp)}
   \end{equation}
as integral cycles in $V_2(\R^4) \subset \S^3 \times \S^3$, where $\Lambda_S : S \to V_2(\R^4)$ is the Legendrian Gauss map of the standard immersion\footnote{If $S$ degenerates to a point $p$, $(\Lambda_S)_\# \DBrack{S}$ has to be interpreted as $\DBrack{\{ p \} \times (\S^3 \cap p^\perp)}$; in general, if $S = \del B_r(p)$, then \[(\Lambda_S)_\# \DBrack{S} = \DBrack{\{ (p \cos r + z \sin r, -p \sin r + z \cos r) : z \in \S^3 \cap p^\perp \}}.\].} $S \into \S^3$ and $\S^3 \cap p_k^\perp$ is the unit sphere of $p_k^\perp = T_{p_k}\S^3$ with the induced orientation.

  Moreover the associated Radon measures $\| (\Lambda_j)_\# \DBrack{\Sigma} \|$ on $V_2(\R^4)$ also converge:
  \begin{equation}
    \label{eq:convergence-associated-measures}
    \| (\Lambda_j)_\# \DBrack{\Sigma} \|
    \weakto 
    \| (\Lambda_S)_\# \DBrack{S} \| + \sum_{k=1}^g \| \DBrack{\{ p_k \} \times (\S^3 \cap p_k^\perp)} \|.
  \end{equation}
\end{thm}

\begin{rmk}
  As long as $S$ is nondegenerate, the mass of the limiting current is $4\pi(1+g)$, therefore no mass is lost at the limit and \eqref{eq:convergence-associated-measures} follows. Instead, if $S$ is degenerate, the limiting current is $-(g-1)(\Lambda_S)_\# \DBrack{S}$ but still \eqref{eq:convergence-associated-measures} keeps track of all the mass.

  This kind of statement is common in problems with lack of strong compactness (most notably, bubbling for harmonic maps, Yang--Mills connections, etc). However, the setting here is much weaker: we can only make sense of this weak limit in the sense of currents and we have no control over how this splitting is realized from the point of view of the surface (for example, the degeneration of its conformal class). This is because minimizing sequences can be potentially very degenerate, as shown in Section~\ref{sec:min-sequences}.
\end{rmk}

\begin{rmk}
  For spheres ($g=0$) the theorem is much easier and follows directly from the fact that the currents $(\Lambda_j)_\# \DBrack{\Sigma}$ are nontrivial in $H_2(\Lambda_2(\R^4), \Z) \simeq \Z$, and the lower bound for the mass in their homology class is precisely $4\pi$. This is easier to see for the Gauss cycles $(G_j)_\# \DBrack{\Sigma}$, which lie in the class $(1, 1) \in \Z^2 \simeq H_2(\S^2 \times \S^2, \Z)$. For higher genus this fails (most drastically when $g = 1$, so that $\Lambda_\# \DBrack{\Sigma}$ and $G_\# \DBrack{\Sigma}$ are homologically trivial) and in order to be able to get compactness we must be able to capture separately the parts of $\Sigma$ converging to the positive sphere and the rest.
\end{rmk}

\subsection{Estimates for almost minimizers}

In order to understand the behavior of almost minimizers of $\AG$ we will need to analyze when an immersion has $\TAC(\phi) \approx \pi \AG(\phi)$ and when it is almost tight. Therefore we will examine in more detail the elements of the proof of \cref{thm:area-gauss-bound}.

Recall that we defined the angles $\tilde \theta_i(x) \in (0, \pi)$ such that $\cot \tilde \theta_i(x) = \kappa_i(x)$, and we obtained the following formula:
\[
  \TAC(\phi) = \int_\Sigma \sqrt{1 + \kappa_1(x)^2} \sqrt{1 + \kappa_2(x)^2} \int_{\S^1} |\sin( \tilde \theta_1(x) - \theta) \sin( \tilde \theta_2(x) - \theta)| \dd \theta \dvol_{\Sigma}(x).
\]

Let us order these angles in such a way that $\tilde \theta_2(x) \leq \tilde \theta_1(x)$, and then define $\theta_2(x) := \tilde \theta_2(x)$ and $\theta_1(x) := \tilde \theta_1(x) - \pi$. The significance of this definition is that $\theta_1(x) < 0 < \theta_2(x)$ for every $x$, and $(\theta_1(x), \theta_2(x))$ is precisely the longest interval containing $\theta = 0$ such that the integrand
\[
  \sin( \tilde \theta_1(x) - \theta) \sin( \tilde \theta_2(x) - \theta)
  = \sin( \theta - \theta_1(x)) \sin( \theta_2(x) - \theta)
\]
is positive. We will also denote $\hat \theta(x) := \theta_2(x) - \theta_1(x) \in (0, \pi]$. In order to avoid writing every time the Jacobian $\Jac(G)$, we introduce the following \emph{Gauss area measure} $\mu$ on $\Sigma$:
\begin{equation}
  \mu(A) := \AG(\phi |_A) = \int_{A} \Jac(G_\phi)(x) \, \dd \vol_{\Sigma}(x)
  \qquad \text{for every } A \subset \Sigma \text{ Borel.}
\end{equation}
Moreover, the following two quantities appear naturally when computing $\TAC(\phi)$:
\[
  \lambda_0(\hat \theta) := \sin (\hat \theta) - \hat \theta \cos (\hat \theta)
  \qquad \text{and} \qquad
  \lambda_\pi(\hat \theta) := \sin (\hat \theta) + (\pi - \hat \theta) \cos (\hat \theta).
\]
  Slightly abusing notation, depending on the context we will consider them as functions of $\hat \theta \in (0, \pi]$ or as functions of $x \in \Sigma$ via $\hat \theta(x) = \theta_2(x) - \theta_1(x)$. We have:
\begin{lemma}
  \label{lem:integrals-lambda}
  \[
    \int_{\theta_1}^{\theta_2} |\sin(\theta - \theta_1 ) \sin(\theta_2 - \theta)| \, \dd \theta
    = \int_{\theta_1}^{\theta_2} \sin(\theta - \theta_1 ) \sin(\theta_2 - \theta) \, \dd \theta
    = \frac{1}{2} \lambda_0(\hat \theta)
  \]
  and
  \[
    \int_{\theta_2}^{\theta_1 + \pi} |\sin(\theta - \theta_1 ) \sin(\theta_2 - \theta)| \, \dd \theta
    = -\int_{\theta_2}^{\theta_1 + \pi} \sin(\theta - \theta_1 ) \sin(\theta_2 - \theta) \, \dd \theta
    = \frac{1}{2} \lambda_\pi(\hat \theta).
  \]
  \begin{proof}
    We just compute:
    \begin{align*}
      &\int_{\theta_1}^{\theta_2} \sin(\theta - \theta_1 ) \sin(\theta_2 - \theta) \, \dd \theta
  = \int_{0}^{\theta_2 - \theta_1} \sin(\varphi) \sin(\theta_2 - \theta_1 - \varphi) \, \dd \varphi \\
  &\quad = \int_{0}^{\hat \theta} \sin(\varphi) \sin(\hat \theta - \varphi) \, \dd \varphi
  = \frac{1}{2} \int_{0}^{\hat \theta} \cos (2 \varphi - \hat \theta) - \cos(\hat \theta) \, \dd \varphi
  = \frac{1}{2} \left( \sin (\hat \theta) - \hat \theta \cos (\hat \theta) \right)
    \end{align*}
    and
    \begin{align*}
      \pushQED{\qed}
      &-\int_{\theta_2}^{\theta_1 + \pi} \sin(\theta - \theta_1 ) \sin(\theta_2 - \theta) \, \dd \theta
  = -\int_{0}^{\theta_1 + \pi - \theta_2} \sin(\theta_2 + \varphi - \theta_1 ) \sin(- \varphi) \, \dd \varphi \\
  &\quad = \int_{0}^{\pi - \hat \theta} \sin(\hat \theta + \varphi) \sin(\varphi) \, \dd \varphi
  = \frac{1}{2} \int_{0}^{\pi - \hat \theta} \cos(\hat \theta) - \cos(\hat \theta + 2 \varphi) \, \dd \varphi \\
  &\quad = \frac{1}{2} (\pi - \hat \theta) \cos (\hat \theta) - \frac{1}{4} ( \sin(\hat \theta + 2 (\pi - \hat \theta))  - \sin(\hat \theta))
  = \frac{1}{2} \left( \sin(\hat \theta) + (\pi - \hat \theta) \cos (\hat \theta) \right).
  &\qedhere
    \end{align*}
  \end{proof}
\end{lemma}

\begin{lemma}
  We have the following identities:
  \begin{equation}
    \label{eq:tac-lambda}
    \int_{\Sigma} \lambda_0 + \lambda_\pi \, \dd \mu = \TAC(\phi);
  \end{equation}
  \begin{equation}
    \label{eq:degree-lambda}
    \int_{\Sigma} \lambda_0 - \lambda_\pi \, \dd \mu = 4\pi^2 (1 - g)
  \end{equation}
  \begin{proof}
    The proof of \eqref{eq:tac-lambda} is immediate from the definition of total absolute curvature (recall \cref{def:total-absolute-curvature}) and \cref{lem:integrals-lambda}. To prove \eqref{eq:degree-lambda}, one can either use a Morse-theoretic argument or just invoke Gauss--Bonnet after undoing the computations in the proof of \cref{thm:area-gauss-bound}:
    \begin{align*}
      \pushQED{\qed}
    \int_{\Sigma} \lambda_0 - \lambda_\pi \, \dd \mu
    &= \int_{\Sigma} \int_0^{2\pi} \sin(\theta - \theta_1) \sin(\theta_2 - \theta) \, \dd \theta \dd \mu \\
    &= \int_{\Sigma} \int_0^{2\pi} (\cos \theta - \kappa_1 \sin \theta)(\cos \theta - \kappa_2 \sin \theta) \, \dd \theta \dd \vol_\Sigma \\
    &= \int_{\Sigma} \pi (1 + \kappa_1 \kappa_2) \dd \vol_\Sigma
    = 4\pi^2(1 - g).
    \qedhere
  \end{align*}
  \end{proof}
\end{lemma}

The following lemma is immediate to check by Taylor expansions and elementary arguments.
\begin{lemma}
  The functions $\lambda_0, \lambda_\pi$ satisfy the following properties:
  \begin{enumerate}[label=(\roman*)]
    \item $\lambda_\pi(\hat \theta) = \lambda_0(\pi - \hat \theta)$;
    \item $\lambda_0$ is strictly increasing $[0, \pi]$, with range $[0, \pi]$;
    \item \label{it:dist-from-pi} $\pi - (\lambda_0 + \lambda_\pi) \geq c \min(\lambda_0, \lambda_\pi) \geq 0$ for a constant $c > 0$;
    \item \label{it:control-lambda-small} $c \hat \theta^2 \leq \lambda_0(\hat \theta)$ and $c (\pi - \hat \theta)^2 \leq \lambda_\pi(\hat \theta)$ for a constant $c > 0$.
  \end{enumerate}
\end{lemma}

Next we introduce the following decomposition $\Sigma = \Sigma_0 \, \dot{\cup} \, \Sigma_\pi$ with
\[
  \Sigma_0 := \{ \lambda_0 \leq \lambda_\pi \} = \{ x \in \Sigma : \hat \theta(x) \leq \pi / 2 \}
\]
and
\[
  \Sigma_\pi := \{ \lambda_0 > \lambda_\pi \} = \{ x \in \Sigma : \hat \theta(x) > \pi / 2 \}.
\]
\begin{lemma}
  \label{lem:estimates-Sigma-pi}
  If $\Sigma$ has genus $g$ and $\AG(\phi) \leq 4 \pi (1+g) + \delta$ for some $0 < \delta < 1$, then
  \begin{equation}
    \label{eq:sigma-pi-area-lb}
    \mu(\Sigma_\pi)
    \geq 4\pi - C \delta.
  \end{equation}
  and
  \begin{equation}
    \label{eq:hat-theta-near-pi}
    \int_{\Sigma_\pi} (\pi - \hat \theta(x)) \, \dd \mu(x)
    \leq C \sqrt{\delta}.
  \end{equation}
  \begin{proof}
    Combining \eqref{eq:tac-lambda} and \eqref{eq:degree-lambda} we obtain
    \begin{equation}
      \label{eq:lambda0}
      \int_{\Sigma} \lambda_0  \, \dd \mu = \frac{\TAC(\phi) + 4\pi^2 (1-g)}{2}
    \end{equation}
    and
    \begin{equation}
      \label{eq:lambdapi}
      \int_{\Sigma} \lambda_\pi \, \dd \mu = \frac{\TAC(\phi) - 4\pi^2 (1-g)}{2}.
    \end{equation}
    On one hand we compute, using \eqref{eq:tac-lambda},
    \[
      \int_\Sigma (\pi - (\lambda_0 + \lambda_\pi)) \, \dd \mu
      = \pi \AG(\phi) - \TAC(\phi)
      \leq \pi \cdot 4\pi(g+1) + \pi \delta - 2\pi^2 (2g + 2)
      = \pi \delta,
    \]
    and on the other hand, using \ref{it:dist-from-pi},
    \[
      \int_\Sigma (\pi - (\lambda_0 + \lambda_\pi)) \, \dd \mu
      \geq c \int_\Sigma \min(\lambda_0, \lambda_\pi) \, \dd \mu.
    \]
    Using our decomposition of $\Sigma$ it follows that
    \begin{equation}
      \label{eq:upper-bound-deviation-0-pi}
      C \delta 
      \geq \int_\Sigma \min(\lambda_0, \lambda_\pi) \, \dd \mu
      = \int_{\Sigma_0} \lambda_0 \, \dd \mu + \int_{\Sigma_\pi} \lambda_\pi \, \dd \mu.
    \end{equation}
    Now, using \eqref{eq:lambda0}, \eqref{eq:upper-bound-deviation-0-pi} and the Chern--Lashof inequality (\cref{thm:chern-lashof}),
    \begin{align*}
      \pi \mu(\Sigma_\pi)
      = \int_{\Sigma_\pi} \pi \, \dd \mu
  &\geq \int_{\Sigma_\pi} \lambda_0 \, \dd \mu
  = \int_{\Sigma} \lambda_0 \, \dd \mu - \int_{\Sigma_0} \lambda_0 \, \dd \mu \\
   &\geq \frac{\TAC(\phi) + 4\pi^2 (1-g)}{2} - C \delta \\
  &\geq \frac{4\pi^2(1+g) + 4\pi^2 (1-g)}{2} - C \delta
  = 4\pi^2 - C \delta
    \end{align*}
    and we deduce \eqref{eq:sigma-pi-area-lb}.
    Again from \eqref{eq:upper-bound-deviation-0-pi} and \ref{it:control-lambda-small} we obtain
    \[
      \int_{\Sigma_\pi} (\pi - \hat \theta(x))^2 \, \dd \mu(x)
      \leq C \int_{\Sigma_\pi} \lambda_\pi \, \dd \mu
      \leq C \delta
    \]
    and using Cauchy--Schwarz we show \eqref{eq:hat-theta-near-pi}.
  \end{proof}
\end{lemma}

It will be key for the proof of \cref{thm:gauss-main} to work with another interval of angles:
for each point $x \in \Sigma$, define $\theta^+(x)$ as the maximum $t > 0$ such that $\dist_{\S^3}(\phi(x) \cos t + \nu(x) \sin t, \phi(\Sigma)) = t$. Clearly this equality holds for small $t$, so $\theta^+(x) > 0$. Define analogously $\theta^-(x) < 0$ so that $\dist_{\S^3}(\phi(x) \cos t + \nu(x) \sin t, \phi(\Sigma)) = -t$ for all $0 \geq t \geq \theta^-(x)$.

\begin{lemma}
  \label{lem:focal-unstable}
  For each $x \in \Sigma$, $\theta_1(x) \leq \theta^-(x) < 0 < \theta^+(x) \leq \theta_2(x) \leq \theta_1(x) + \pi$.
  \begin{proof}
    It is enough to show that $\theta^+(x) \leq \theta_2(x)$, since the first inequality is symmetric and the rest has already been established. It suffices to prove that for any $\theta_2(x) < \ell < \pi$, the geodesic $\gamma : t \in [0, \ell] \mapsto \phi(x) \cos t + \nu(x) \sin t$ is unstable with respect to perturbations $\tilde \gamma(t, s)$ with $\tilde \gamma(\ell, s) \equiv \gamma(\ell)$ and such that $\tilde \gamma(0, s) \in \phi(\Sigma)$.

    Recall that the second variation formula is
    \begin{equation}
      \label{eq:second-variation-geodesic}
      \left. \frac{\dd^2}{\dd s^2} \right|_{s = 0} \operatorname{Len}(\tilde \gamma(\cdot, s))
        = \int_{0}^\ell |V'(t)|^2 - |V(t)|^2 \, \dd t + \left[ \left\langle \left. \frac{\del^2}{\del s^2} \right|_{s=0} \tilde \gamma(t, s), \gamma'(t) \right \rangle \right]_0^\ell,
    \end{equation}
    where $V(t) = \left. \tfrac{\del}{\del s} \right|_{s=0} \tilde \gamma(t, s)$ and $V'$ denotes the covariant derivative along $\gamma$. In our case $\tilde \gamma(\ell, s)$ is constant and $\tilde \gamma(0, s)$ describes a curve in $\phi(\Sigma)$, so that
      \[
        \left\langle \left. \frac{\dd^2}{\dd s^2} \right|_{s=0} \tilde \gamma(0, s), \gamma'(0) \right \rangle
          = \langle \gamma'(0), A_{\phi(x)}(V(0), V(0)) \rangle.
      \]
      Now let $e \in T_{\phi(x)}\Sigma$ denote a principal direction for the curvature $\kappa_2(x)$, and extend it to a parallel vector field $e(t)$ along $\gamma$. Choose a variation $\tilde \gamma$ with vector field $V(t) = f(t) e(t)$, where $f(t) = \sin(\ell - t)$, and plug it into \eqref{eq:second-variation-geodesic}:
      \begin{align*}
        \pushQED{\qed}
        \left. \frac{\dd^2}{\dd s^2} \right|_{s = 0} \operatorname{Len}(\tilde \gamma(\cdot, s))
          &= \int_{0}^\ell f'(t)^2 - f(t)^2 \, \dd t - f(0)^2 \langle \nu(x), A_{\phi(x)}(e, e) \rangle \\
          &= \int_{0}^\ell \cos(\ell - t)^2 - \sin(\ell - t)^2 \, \dd t - \kappa_2 \sin(\ell)^2 \\
          &= \sin \ell \cos \ell - \kappa_2 \sin^2 \ell \\
          &= \sqrt{1 + \kappa_2^2} \sin (\ell) \sin (\theta_2 - \ell) < 0.
          \qedhere
        \end{align*}
  \end{proof}
\end{lemma}

Moreover we have the following fact about $\theta^\pm$:
\begin{lemma}
  \label{lem:min-distance}
  Let $p \in \S^3$. Then the function $\langle p, \phi \rangle : \Sigma \to \R$ attains a minimum at $x \in \Sigma$ if and only if there exists $t \in [\theta^-(x), \theta^+(x)]$ such that $p = \phi(x) \cos t + \nu(x) \sin t$.
  \begin{proof}
    It is clear that
    \[
      \dist_{\S^3}(p, \phi(y)) = \cos^{-1}(\langle p, \phi(y) \rangle).
    \]
    If $p = \phi(x) \cos t + \nu(x) \sin t$ for some $x \in \Sigma$ and some $t \in [\theta^-(x), \theta^+(x)]$, then by definition the left hand side attains its minimum at $x$, so $\langle p, \phi(y) \rangle$ attains its maximum there.

    Conversely, suppose that $\dist_{\S^3}(p, \phi(\cdot))$ attains its minimum at some point $x$, and let $\gamma(t) := \phi(x) \cos t + \nu(x) \sin t$ be the geodesic starting at $\phi(x)$ realizing this minimum (it has this form because it must meet $\phi(\Sigma)$ orthogonally). If $t_0 > 0$ is the first time at which $p = \gamma(t_0)$, then $\dist(p, \phi(\Sigma)) = t_0$ and therefore $t_0 \in [\theta^-(x), \theta^+(x)]$ as desired.
  \end{proof}
\end{lemma}

Now denote
\[
  E := \{ (x, \theta) \in \Sigma \times (-\pi, \pi) : \theta^-(x) \leq \theta \leq \theta^+(x) \}
\]
As a corollary of \cref{lem:min-distance} we have:
\begin{cor}
  \label{cor:surjectivity-E}
  The restriction of the map $F$ from \eqref{eq:def-F} to $E$ is surjective, that is, $F(E) = \S^3$.
  \begin{proof}
    For any $p \in \S^3$, let $x \in \Sigma$ attain the minimum of $\langle p, \phi(\cdot) \rangle$. Then by \cref{lem:min-distance} $\exists \theta \in [\theta^-(x), \theta^+(x)]$ such that $F(x, \theta) = p$ and of course $(x, \theta) \in E$.
  \end{proof}
\end{cor}

Recall that we had the inclusions $(\theta^-(x), \theta^+(x)) \subseteq (\theta_1(x), \theta_2(x))$ for every $x \in \Sigma$. We now show that the difference of these two intervals is, on average, small in $\Sigma_\pi$:
\begin{lemma}
  If $\AG(\phi) \leq 4\pi(1+g) + \delta$, then
  \begin{equation}
    \label{eq:interval-difference-sigma-pi}
      \int_{\Sigma_\pi} |\theta^-(x) - \theta_1(x)| + |\theta_2(x) - \theta^+(x)| \dd\mu(x)
      \leq C \delta^{1/3}
  \end{equation}
  for a constant $C = C(g)$.
  \begin{proof}
    By \cref{lem:integrals-lambda},
    \begin{align*}
      2 \int_{\Sigma} \int_{\theta_1(x)}^{\theta_2(x)} \sin(\theta - \theta_1(x)) \sin(\theta_2(x) - \theta) \dd \theta \dd\mu(x)
      = \int_{\Sigma} \lambda_0(\hat \theta(x)) \dd \mu(x),
    \end{align*}
    and by \eqref{eq:lambda0} and \eqref{eq:area-tac-genus-bound},
    \[
      \int_{\Sigma} \lambda_0 \dd \mu
      = \frac{\TAC(\phi) + 4\pi^2 (1-g)}{2}
      \leq \frac{4\pi^2(1+g) + \pi \delta + 4\pi^2 (1-g)}{2}
      = 4\pi^2 + \frac{1}{2} \pi \delta.
    \]
    On the other hand, thanks to \cref{cor:surjectivity-E},
    \begin{align*}
      \int_{\Sigma} \int_{\theta^-(x)}^{\theta^+(x)} \sin(\theta - \theta_1(x)) \sin(\theta_2(x) - \theta) \dd \theta \dd\mu(x)
      = \int_{E} \Jac(F) \dd \theta \dvol_{\Sigma}(x)
      = 2\pi^2.
    \end{align*}
    Therefore, putting these two estimates together we get
    \begin{align*}
      \int_{\Sigma} \int_{[\theta_1(x), \theta_2(x)] \setminus [\theta^-(x), \theta^+(x)]} \sin(\theta - \theta_1(x)) \sin(\theta_2(x) - \theta) \dd \theta \dd\mu(x)
      &\leq \frac{1}{4} \pi \delta,
    \end{align*}
    or equivalently
    \begin{equation}
      \label{eq:estimate-difference-intervals}
      \int_{\Sigma} \int_{[\theta_1(x), \theta^-(x)] \cup [\theta^+(x), \theta_2(x)]} \sin(\theta - \theta_1(x)) \sin(\theta_2(x) - \theta) \dd \theta \dd\mu(x)
      \leq \frac{1}{4} \pi \delta.
    \end{equation}
    We now claim that there is a constant $c > 0$ such that, for $x \in \Sigma_\pi$,
    \begin{equation}
      \label{eq:small-difference-intervals}
      \int_{[\theta_1, \theta^-] \cup [\theta^+, \theta_2]} \sin(\theta - \theta_1) \sin(\theta_2 - \theta) \dd \theta 
      \geq c \left((\theta^- - \theta_1)^3 + (\theta_2 - \theta^+)^3 \right).
    \end{equation}
    By symmetry it is enough to prove the estimate for one of the two intervals, so let $\ell$ be the length of $[\theta_1(x), \theta^-(x)]$. Suppose first that $\ell \leq \pi / 4$ and recall that $\hat \theta(x) = \theta_2(x) - \theta_1(x) \geq \pi / 2$. We use the change of variable
    \[
      \int_{\theta_1(x)}^{\theta^-(x)} \sin(\theta - \theta_1(x)) \sin(\theta_2(x) - \theta) \dd \theta 
      = \int_{0}^{\ell} \sin(\varphi) \sin(\hat \theta - \varphi) \dd \varphi.
    \]
    Here $\pi \geq \hat \theta - \varphi \geq \pi / 2 - \ell \geq \pi / 4$ and hence $\sin(\hat \theta - \varphi) \geq c (\pi - (\hat \theta - \varphi)) = c (\pi - \hat \theta + \varphi) \geq c \varphi$ for a constant $c > 0$. Also $\sin \varphi \geq c \varphi$ for $0 \leq \varphi \leq \pi / 4$, and thus
    \[
      \int_{0}^{\ell} \sin(\varphi) \sin(\hat \theta - \varphi) \dd \varphi
      \geq c \int_{0}^{\ell} \varphi^2 \dd \varphi = c \ell^3.
    \]
    On the other hand, if $\ell > \pi / 4$, then by the above estimate
    \[
      \int_{0}^{\ell} \sin(\varphi) \sin(\hat \theta - \varphi) \dd \varphi
      \geq \int_{0}^{\pi / 4} \sin(\varphi) \sin(\hat \theta - \varphi) \dd \varphi
      \geq c (\pi / 4)^3
      \geq c \ell^3
    \]
    since $\ell \leq \pi$. Putting \eqref{eq:estimate-difference-intervals} and \eqref{eq:small-difference-intervals} together, we have
    \[
      \int_{\Sigma_\pi} (\theta^-(x) - \theta_1(x))^3 + (\theta_2(x) - \theta^+(x))^3 \dd\mu(x)
      \leq C \delta
    \]
    and \eqref{eq:interval-difference-sigma-pi} follows by H\"older's inequality.
  \end{proof}
\end{lemma}

\subsection{Uniform convergence to a round sphere}
Thanks to the estimates of the previous subsection, we may now deduce a first geometric consequence:

\begin{cor}
  \label{cor:convergence-round-sphere}
  For every $\eps > 0$ and $g \in \N$ there exists $\delta > 0$ such that, if $\phi : \Sigma \to \S^3$ is an immersion of a surface of genus $g$ with $\AG(\phi) \leq 4\pi(1+g) + \delta$, then there exists a round sphere $S \subset \S^3$ such that $\phi(\Sigma)$ is contained in the $\eps$-neighborhood of $S$.
  \begin{proof}
    Observe that
    \begin{align*}
      \pi - (\theta^+(x) - \theta^-(x))
      &\leq \pi - (\theta_2(x) - \theta_1(x)) + (\theta_2(x) - \theta^+(x)) + (\theta^-(x) - \theta_1(x)) \\
      &= (\pi - \hat \theta(x)) + (\theta_2(x) - \theta^+(x)) + (\theta^-(x) - \theta_1(x)).
    \end{align*}
    Then, using \eqref{eq:interval-difference-sigma-pi} and \eqref{eq:hat-theta-near-pi} we have
    \begin{align*}
      \int_{\Sigma_\pi} \pi - (\theta^+(x) - \theta^-(x)) \, \dd \mu(x)
      &\leq \int_{\Sigma_\pi} |\pi - \hat \theta(x)| \, \dd \mu(x) \\
      &\quad+ \int_{\Sigma_\pi} |\theta_2(x) - \theta^+(x)| + |\theta^-(x) - \theta_1(x)| \, \dd \mu(x) \\
      &\leq C \delta^{1/2} + C \delta^{1/3}
      \leq C \delta^{1/3},
    \end{align*}
    and hence together with the lower bound \eqref{eq:sigma-pi-area-lb},
    \begin{equation}
      \label{eq:average-geodesic-distance}
      \frac{\int_{\Sigma_\pi} \pi - (\theta^+(x) - \theta^-(x)) \, \dd \mu(x)}{\mu(\Sigma_\pi)} \leq C \delta^{1/3}.
    \end{equation}
    In particular there exists $x_0 \in \Sigma_\pi$ such that $\pi - (\theta^+(x_0) - \theta^-(x_0)) \leq C \delta^{1/3}$. Now letting $p_\pm := \phi(x_0) \cos (\theta^\pm(x_0)) + \nu(x_0) \sin(\theta^\pm(x_0))$, we have that the geodesic balls with centers $p_\pm$ and radii $r_\pm := |\theta^\pm(x_0)|$ are disjoint from each other and lie in $\S^3 \setminus \phi(\Sigma)$.

    Since $r_+ + r_- \geq \pi - C \delta^{1/3}$, it is clear that by making $\delta$ small enough we can guarantee that
    \[
      \phi(\Sigma)
      \subset \S^3 \setminus \left( B_{r_+}(p_+) \cup B_{r_-}(p_-) \right)
      \subset B_\eps(S),
    \]
    where $S = \del B_{r_+}(p_+)$.
  \end{proof}
\end{cor}

Even for minimizing sequences of tori \emph{immersed} in $\R^3$,
we cannot expect in general the normal map to be surjective onto $\S^2$ (as in the torus of revolution from a figure-8 curve), and it is hard to conjecture the general behavior of Gauss maps. Arguing as in \cref{prop:concentrating-sequence}, we cannot expect the behavior in $\S^3$ to be any better than that.
In view of this observation, for the rest of this section we will work with a sequence of embeddings $\phi_j : \Sigma \to \S^3$ ($j = 1, 2, \ldots$), whose corresponding normals $\nu_j$ and Legendrian Gauss maps $\Lambda_j = (\phi_j, \nu_j)$ satisfy
\begin{equation}
  \label{eq:areas-G-minimizing}
  \AG(\phi_j) = \Area(\Lambda_j) \searrow 4\pi(1+g).
\end{equation}

By the embeddedness assumption, $\S^3 \setminus \phi_j(\Sigma)$ consists of two components which we will label as $\Omega^\pm_{j}$, with the normal $\nu_j$ pointing towards $\Omega^+_{j}$. This allows us to give a slightly stronger conclusion than \cref{cor:convergence-round-sphere}:
\begin{lemma}
  \label{lem:hausdorff-convergence}
  If $(\phi_j)$ is a sequence of embeddings that satisfies \eqref{eq:areas-G-minimizing}, then after taking a subsequence there is a (unique) possibly degenerate round $2$-sphere $S \subset \S^3$ such that
  \begin{equation}
    \label{eq:hausdorff-convergence}
    \phi(\Sigma_j) \xrightarrow{j\to\infty} S
    \qquad \text{in the Hausdorff distance.}
  \end{equation}
  \begin{proof}
    Let $x_j \in \Sigma$ attain the maximum of $\theta^+_{j} - \theta^-_{j}$ (the corresponding parameters for $\Sigma_j$, and write $r^\pm_{j} := |\theta^\pm_{j}(x_j)|$. By \eqref{eq:areas-G-minimizing} and \eqref{eq:average-geodesic-distance}, $r^+_{j} + r^-_{j} \nearrow \pi$.
    Letting
    \[
      p^+_{j} := \phi_j(x_j) \cos r^+_{j} + \nu_j(x_j) \sin r^+_{j}
      \quad \text{and} \quad
      p^-_{j} := \phi_j(x_j) \cos r^-_{j} - \nu_j(x_j) \sin r^-_{j},
    \]
    we have that the two open geodesic balls $B^+_{j} := B_{r^+_{j}}(p^+_{j})$ and $B^-_{j} := B_{r^-_{j}}(p^-_{j})$ are disjoint from $\phi_j(\Sigma)$ and hence $B^\pm_{j} \subseteq \Omega^\pm_{j}$ (as this inclusion holds near $\phi_j(x_j)$). We have that
    \[
      \dist_\Haus \left(\S^3 \setminus (B^+_{j} \cup B^-_{j}), \del B^+_{j} \right)
      \xrightarrow{j \to \infty} 0,
    \]
    and also
    \[
      \dist_\Haus \left(\S^3 \setminus (B^+_{j} \cup B^-_{j}), \phi_j(\Sigma) \right)
      \xrightarrow{j \to \infty} 0
    \]
    because, for every point $p \in \S^3 \setminus (B^+_{j} \cup B^-_{j})$, the shortest curve that joins $B^\pm_{j}$ through $p$, which has length $o(1)$, must cross $\phi_j(\Sigma)$. This gives
    \[
      \dist_\Haus \left(\phi_j(\Sigma), \del B^+_{j} \right)
      \xrightarrow{j \to \infty} 0.
    \]
    Finally, since the space of (possibly degenerate) round $2$-spheres in $\S^3$ is compact in the Hausdorff distance, after extracting a subsequence we also have that $\dist_\Haus(\del B^+_{j}, S) \to 0$ for some $S$, and the result follows.
  \end{proof}
\end{lemma}

After a change of coordinates, we may assume that $S = \del B_R(p_S)$ where $p_S = -e_4 \in \S^3$ is the South pole and $0 \leq R \leq \pi / 2$. Moreover, up to reorienting $\Sigma$, we may assume that (for $j$ large enough) $\Omega^+_{j}$ contains the North pole $p_N = e_4$.
We also have:
\begin{lemma}
  \label{lem:most-sigmapi-goes-far}
  For every $\eta > 0$ there exists $j_0 \in \N$ such that, if $j \geq j_0$, then
  \[
    \mu \left(\left\{ x \in \Sigma^{(j)}_\pi : \theta^+(x) > \pi - R - \eta \text{ and } |\theta^-(x)| > R - \eta \right\}\right) \geq 4\pi - \eta.
  \]
  \begin{proof}
    From \eqref{eq:average-geodesic-distance} we have that
    \[
      \mu_j \left( \left\{x \in \Sigma_\pi^{(j)} : \pi - (\theta^+(x) - \theta^-(x)) > \frac{\eta}{2} \right\} \right)
      \leq \frac{2}{\eta} \int_{\Sigma_\pi^{(j)}} \pi - (\theta^+(x) - \theta^-(x)) \, \dd \mu_j(x)
      \leq \frac{\eta}{2}
    \]
    if $j$ is large enough. Together with \eqref{eq:sigma-pi-area-lb}, we will be done if we show that points in the complement of this set satisfy the required conclusion. Thus, let $x \in \Sigma$ such that $\pi - \tfrac{\eta}{2} < \theta^+(x) - \theta^-(x)$. Assuming for a contradiction that $\theta^+(x) \leq \pi - R - \eta$, then
    \[
      \pi - \frac{\eta}{2} < \theta^+(x) + |\theta^-(x)| \leq \pi - R - \eta + |\theta^-(x)|,
    \]
    so $|\theta^-(x)| > R + \frac{\eta}{2}$. By following the normal geodesic at $\phi_j(x)$ starting from $-\nu_j(x)$ for a time $|\theta^-(x)|$, we find a point $p$ in $\Omega^-_{j}$ whose distance to $\phi_j(\Sigma)$ is at least $R + \frac{\eta}{2}$. Then, using \eqref{eq:hausdorff-convergence}, we conclude that, provided that $j$ is large enough, the distance from $p$ to $\del B_R(p_S)$ is strictly bigger than $R$. As a result, $p$ cannot be in $B_R(p_S)$ and we must have $p \in \S^3 \setminus B_{2R}(p_S)$. Clearly for $j$ large enough this implies that $p$ belongs to $\Omega^+_{j}$, a contradiction.

    On the other hand, if $|\theta^-(x)| \leq R - \eta$, then
    \[
      \pi - \frac{\eta}{2} < \theta^+(x) + |\theta^-(x)| \leq R - \eta + \theta^+(x),
    \]
    so $\theta^+(x) > \pi - R + \frac{\eta}{2}$ and, arguing as before, for $j$ large we get a point $p \in \S^3$ at distance at least $\pi - R \geq R$ from $\del B_R(p_S)$, which is impossible.
  \end{proof}
\end{lemma}

\subsection{Decomposition of the normal cycle}

The next step in the proof is to decompose the integral cycles $(\Lambda_j)_\# \DBrack{\Sigma}$ into two pieces $N_j^+$ and $N_j^-$, both integral Legendrian cycles representing nontrivial classes in homology, which converge weakly to homology minimizers, and such that $\Mass(N_j^+) + \Mass(N_j^-) - \AG(\phi_j) \to 0$. We will construct the cycle $N_j^+$ as the normal cycle of a surface parallel to a convex set that captures most of the positive curvature of $\phi_j(\Sigma)$.

\begin{prop}
  \label{prop:convex-sets}
  There exist sequences of positive numbers $\delta_j \to 0$ and $\rho_j \to \pi - R$, and a sequence of geodesically convex sets $K_j \subset B_{\delta_j}(p_N) \subset \S^3$, such that denoting
  \[
    \Sigma_j^+ := \{ x \in \Sigma^{(j)}_\pi : \theta^+(x) > \rho_j \text{ and } \theta^-(x) > R / 2 \} \subset \Sigma
  \]
  and introducing the parallel map $P_j : \Sigma_j^+ \to \S^3$,
\[
  P_j(x) := \phi_j(x) \cos \rho_j + \nu_j(x) \sin \rho_j,
\]
  we have that $\mu_j(\Sigma_j^+) \geq 4\pi - \delta_j$ and $P_j$ is an embedding of $\Sigma_j^+$ into an open subset of $\del K_j$.

  \begin{proof}
    By \cref{lem:most-sigmapi-goes-far}, we can find a sequence $\eps_j \to 0$ such that, with $\rho_j := \pi - R - \eps_j$ and the subsets $\Sigma_j^+$ defined above, we have $\mu(\Sigma_j^+) \geq 4\pi - \eps_j$.

    We distinguish two cases: first, if $R < \pi / 2$, define
    \begin{equation}
      K_j := \{ p \in \Omega^+_{j} : \dist(p, \phi_j(\Sigma)) \geq \rho_j \} \subset \S^3.
    \end{equation}
    Observe that, for large $j$, the sets $K_j$ are geodesically convex:
    \[
      K_j
      = \bigcap_{x \in \Sigma} \{ p \in \Omega^+_{j} : \dist(p, \phi_j(x)) \geq \rho_j \}
      = \bigcap_{x \in \Sigma} \{ p \in \Omega^+_{j} : \dist(p, -\phi_j(x)) \leq \pi - \rho_j \},
    \]
    which is the intersection of geodesically convex sets because eventually $\pi - \rho_j < \pi / 2$. Moreover, for each $p \in K_j$,
    \[
      \dist(p, S)
      \geq \dist(p, \phi_j(\Sigma)) - \dist_\Haus(\phi_j(\Sigma), S)
      \geq \rho_j - o_j(1)
      = \pi - R - o_j(1),
    \]
    therefore $K_j \subseteq B_{\delta_j}(p_N)$ and also $\mu(\Sigma_j^+) \geq 4\pi - \delta_j$ for a sequence $\delta_j \to 0$.

  In the case $R = \pi / 2$ we have to work a bit harder: let $Z_j := \{ q \in \Omega_{j}^- : \dist(q, \phi_j(\Sigma)) = 2 \eps_j \}$ and consider instead
  \[
    K_j := \{ p \in \Omega_{j}^+ : \dist(p, Z_j) \geq \rho_j + 2 \eps_j \}.
  \]
  Since now $\rho_j + 2 \eps_j > \pi / 2$ for each $j$, these sets are again convex, and since still $Z_j \to S$ in the Hausdorff distance, the same argument as above shows that $K_j$ is contained in infinitesimally small balls.

  We claim that for every $p \in \Omega_{j}^+$, $\dist(p, Z_j) \geq \dist(p, \phi_j(\Sigma)) + 2 \eps_j$: if $\gamma$ is a geodesic joining $p$ and $Z_j$ at $q$, then it must intersect $\phi_j(\Sigma)$ at some point $a$; it follows that
  \begin{align*}
    \dist(p, Z_j)
    = \operatorname{len}(\gamma)
    \geq \dist(p, a) + \dist(a, q)
    &\geq \dist(p, \phi_j(\Sigma)) + \dist(q, \phi_j(\Sigma)) \\
    &= \dist(p, \phi_j(\Sigma)) + 2 \eps_j.
  \end{align*}
  However, for points $p \in \Omega_{j}^+$ along the geodesic starting at $\phi_j(x)$ with initial tangent vector $\nu_j(x)$, provided that $x \in \Sigma_j^+$ and $\dist(p, \phi_j(\Sigma)) \leq \theta^+(x)$, the opposite inequality is also true: namely, clearly $q := \phi_j(x) \cos (2\eps_j) - \nu_j(x) \sin (2\eps_j) \in Z_j$, and then
  \[
    \dist(p, Z_j)
    \leq \dist(p, q)
    \leq \dist(p, \phi_j(x)) + \dist(\phi_j(x), q)
    = \dist(p, \phi(\Sigma)) + 2\eps_j.
  \]
  In either case, and for $x \in \Sigma_j^+$, by moving slightly up and down in the geodesic, we find points arbitrarily close to each $P_j(x)$ both in $K_j$ and in its complement. This shows that $P_j$ maps $\Sigma_j^+$ into $\del K_j$.

Now we prove the last claims.
    Since $0 < \rho_j < \theta^+(x) < \theta_2(x)$, which is the first focal point in the direction of $\nu_j(x)$, $P_j$ is an immersion on $\Sigma_j^+$. Moreover, $\phi_j(x)$ is the closest point in $\phi_j(\Sigma)$ to $P_j(x)$, which proves that $P_j$ is injective.
    Finally, recall that boundaries of convex sets are Lipschitz surfaces, in particular locally homeomorphic to $\R^2$, so by Brouwer's theorem of invariance of domain, $P_j$ is a local homeomorphism and therefore an embedding.
  \end{proof}
\end{prop}

The theory of curvature measures (see, for example, \cite{morvan-generalized-curvatures} or \cite{rataj-zahle-curvature-measures}) now provides us with a normal cycle for $K_j$, that is, a Legendrian integral $2$-current with zero boundary which lies in the unit tangent bundle of $\S^3$ (which is isomorphic to the Stiefel manifold $V_2(\R^4)$) that generalizes the current carried by the Legendrian Gauss map to sets that are not smooth.

In order to define it and to do computations, it is convenient to regularize $K_j$ by taking a parallel manifold, and then go back by a ``parallel transformation'' in $V_2(\R^4)$. More precisely, for $r \in \R$ define the following parallel map $\mathcal{P}_r$:
\[
  \mathcal{P}_r : (a, b) \in V_2(\R^4) \longmapsto (a \cos r + b \sin r, -a \sin r + b \cos r) \in V_2(\R^4).
\]
This is an isometry, preserves the contact structure, and satisfies $\pi \circ \mathcal{P}_r = \pi$. For a geodesically convex set $K \subset \S^3$ and $r > 0$ denote by $K_r$ the parallel set
\[
  K_r := \{ p \in \S^3 : \dist(p, K) \leq r \}.
\]
It is well known these sets are $C^{1,1}$ for small $r$ (see \cite{federer-curvature-measures}), so that the Legendrian Gauss maps $\Lambda_{K_r} : \del K_r \to V_2(\R^4)$ are well-defined Lipschitz maps, oriented so that the normal vectors $\nu_{K_r}$ point inside $K_r$. Then the currents $(\mathcal{P}_{r})_\# (\Lambda_{K_r})_\# \DBrack{\del K_r}$ are independent of $r$ for $r > 0$ small, and define the normal cycle $N_K$ of $K$. For concreteness, it will be convenient to fix some $0 < \delta_j' < \delta_j$ and define $\tilde K_j := (K_j)_{\delta_j'} \simeq \S^2$ so that
\[
  N_{K_j} = (\mathcal{P}_{\delta_j'})_\# (\Lambda_{\tilde K_j})_\# \DBrack{\del \tilde K_j}
\]
and the map $\tilde P_j : \Sigma_j^+ \to \del \tilde K_j$,
\[
  \tilde P_j(x) := \phi_j(x) \cos (\tilde \rho_j) + \nu_j(x) \sin (\tilde \rho_j),
  \qquad \tilde \rho_j := \rho_j - \delta_j'
\]
is still an embedding for the same reason as $P_j$.
Finally, in order to define the ``positive part'' $N_j^+$ of $(\Lambda_j)_\# \DBrack{\Sigma}$, we simply push forward $N_{K_j}$ by $\mathcal{P}_{-\rho_j}$, and define $N_j^-$ so that $(\Lambda_j)_\# \DBrack{\Sigma} = N_j^+ - N_j^-$ holds:
\begin{equation}
  \label{eq:decomp-N}
  N_j^+ := (\mathcal{P}_{-\rho_j})_\# N_{K_j}
  \qquad \text{and} \qquad
  N_j^- := N_j^+ - (\Lambda_j)_\# \DBrack{\Sigma}.
\end{equation}
Alternatively, $N_j^+ = (\mathcal{P}_{-\tilde \rho_j})_\# (\Lambda_{\tilde K_j})_\# \DBrack{\del \tilde K_j}$.

\begin{lemma}
  For every $x \in \Sigma_j^+$ it holds that $\mathcal{P}_{-\tilde \rho_j}(\Lambda_{\tilde K_j}(\tilde P_j(x))) = \Lambda_j(x)$, and also
\begin{equation}
  \label{eq:pushforward-gauss-measure}
  (\tilde P_j)_\# (\mu_{j} \res \Sigma_j^+) = \mu_{\tilde K_j} \res \tilde P_j(\Sigma_j^+)
  \qquad \text{as measures on } \del \tilde K_j.
\end{equation}

\begin{proof}
  The first claim can be rewritten as $\Lambda_{\tilde K_j}(\tilde P_j(x)) = \mathcal{P}_{\tilde \rho_j}(\Lambda_j(x))$, and this is clear by examining the expression and using the fact that $\tilde P_j$ is an embedding into $\tilde K_j$. The second claim follows from the first by using the area formula and the injectivity of $\tilde P_j$: for any Borel $A \subset \tilde K_j$,
  \begin{align*}
    \pushQED{\qed}
    \left( \mu_{\tilde K_j} \res \tilde P_j(\Sigma_j^+) \right) (A)
    &= \int_{A \cap \tilde P_j(\Sigma_j^+)} \Jac(\Lambda_{\tilde K_j})(y) \, \dd \Haus^2(y) \\
    &= \int_{\tilde P_j^{-1}(A) \cap \Sigma_j^+} \Jac(\Lambda_{\tilde K_j})(\tilde P_j(x)) \cdot \Jac(\tilde P_j)(x) \, \dd \Haus^2(x) \\
    &= \int_{\tilde P_j^{-1}(A) \cap \Sigma_j^+} \Jac(\Lambda_{j})(x) \, \dd \Haus^2(x) \\
    &= (\mu_{j} \res \Sigma_j^+)(\tilde P_j^{-1}(A))
    = (\tilde P_j)_\# (\mu_{j} \res \Sigma_j^+)(A).
    \qedhere
  \end{align*}
\end{proof}
\end{lemma}

We will need the following lemma to estimate the mass of $N_j^+$:

\begin{lemma}
  \label{lem:error-terms-convex}
  For every $\eps > 0$ there exists $\delta > 0$ with the following property: if $V \subset \S^3$ is a closed geodesically convex set contained in $B_\delta(p_N)$ and with $C^{1,1}$ boundary, then
  \begin{equation}
    \Haus^2(\del V) \leq \eps
  \end{equation}
  and
  \begin{equation}
    \int_{\del V} |A| \, \dd \Haus^2 \leq \eps,
  \end{equation}
  where $A$ denotes the second fundamental form of $\del V$.
  \begin{proof}
    Let $\lambda_1, \lambda_2 : \del V \to \R$ denote the principal curvatures of $\del V$, computed with respect to the \emph{inner} normal (so that, following our convention, they are nonnegative). For each $0 < t < \pi / 2 - \delta$ we have
    \begin{align*}
      \vol(V_t)
      &= \vol(V) + \int_{0}^t \int_{\del V} (\cos \theta + \lambda_1 \sin \theta)(\cos \theta + \lambda_2 \sin \theta) \, \dd \Haus^2 \, \dd \theta \\
      &= \vol(V) + \int_{\del V} \int_{0}^t \cos^2 \theta - (\lambda_1 + \lambda_2) \cos \theta \sin \theta + \lambda_1 \lambda_2 \sin^2 \theta \, \dd \theta \, \dd \Haus^2 \\
      &= \vol(V) + \int_{\del V} \left( \frac{t}{2} + \frac{\sin 2t}{4} \right) + (\lambda_1 + \lambda_2) \frac{1 - \cos 2t}{2} + \lambda_1 \lambda_2 \left( \frac{t}{2} - \frac{\sin 2t}{4} \right) \, \dd \Haus^2 \\
      &= \vol(V) + \int_{\del V} \frac{\sin 2t}{2} + (\lambda_1 + \lambda_2) \frac{1 - \cos 2t}{2} + (1 + \lambda_1 \lambda_2) \left( \frac{t}{2} - \frac{\sin 2t}{4} \right) \, \dd \Haus^2 \\
      &= \vol(V) + \int_{\del V} \frac{\sin 2t}{2} + (\lambda_1 + \lambda_2) \frac{1 - \cos 2t}{2} \, \dd \Haus^2 + 4\pi \left( \frac{t}{2} - \frac{\sin 2t}{4} \right) 
    \end{align*}
    On one hand, choose $t = \delta$ and keep only the second term:
    \begin{align*}
      \vol(V_\delta)
      &\geq \frac{\sin 2\delta}{2} \Haus^2(\del V)
      \geq c \delta \Haus^2(\del V).
    \end{align*}
    On the other hand, we have that $V_\delta \subset B_{2\delta}(p_N)$ and therefore
    \[
      \vol(V_\delta) \leq \vol(B_{2\delta}(p_N)) \leq C \delta^3,
    \]
    which gives that $\Haus^2(\del V) \leq C \delta^2$. For the bound on the second fundamental form, we have that $|A| \lesssim |\lambda_1| + |\lambda_2| = \lambda_1 + \lambda_2$, so it is enough to estimate the integral of the mean curvature.
    As before, we have
    \begin{align*}
      \vol(V_t)
      &= \vol(V) + \int_{\del V} \frac{\sin 2t}{2} + (\lambda_1 + \lambda_2) \frac{1 - \cos 2t}{2} \, \dd \Haus^2 + 4\pi \left( \frac{t}{2} - \frac{\sin 2t}{4} \right) \\
      &\geq \frac{1 - \cos 2t}{2} \int_{\del V} (\lambda_1 + \lambda_2) \, \dd \Haus^2 + 2\pi \left( t - \frac{\sin 2t}{2} \right) 
    \end{align*}
    and also $V_t \subseteq B_{t+\delta}(p_N)$, so that
    \[
      \vol(V_t)
      \leq \vol(B_{t+\delta}(p_N))
      = 2\pi \left( (t+\delta) - \frac{\sin(2 (t+\delta))}{2} \right).
    \]
    Now choose, for example, $t = \pi / 4$ and get
    \[
      2\pi \left( \delta + \frac{1 - \cos(2\delta)}{2} \right)
      \geq \frac{1}{2} \int_{\del V} (\lambda_1 + \lambda_2) \, \dd \Haus^2,
    \]
    which gives the desired upper bound by taking $\delta$ small enough.
  \end{proof}
\end{lemma}

As a consequence we have:

\begin{lemma}
  \label{lem:area-gauss-convex-4pi}
  The areas of the Gauss maps of $\tilde K_j$ converge to $4\pi$.
  \begin{proof}
    If $\lambda_1, \lambda_2$ denote the principal curvatures of the $C^{1,1}$ submanifolds $\tilde K_j$, then
    \begin{align*}
      4\pi \leq \AG(\tilde K_j)
      &= \int_{\del \tilde K_j} \sqrt{1 + \lambda_1^2 + \lambda_2^2 + (\lambda_1 \lambda_2)^2} \, \dd \Haus^2
      \leq \int_{\del \tilde K_j} 1 + \lambda_1 + \lambda_2 + \lambda_1 \lambda_2 \, \dd \Haus^2 \\
      &= \int_{\del \tilde K_j} 1 + \lambda_1 \lambda_2 \, \dd \Haus^2 + o_j(1)
      = 4\pi + o_j(1)
    \end{align*}
    thanks to \cref{lem:error-terms-convex} and Gauss--Bonnet.
  \end{proof}
\end{lemma}

\begin{lemma}
  The cycles $N_j^\pm$ from \eqref{eq:decomp-N} satisfy
  \begin{equation}
    \label{eq:decomp-mass}
    \limsup_{j \to \infty} \Mass(N_j^+) \leq 4\pi
    \qquad \text{and} \qquad
    \limsup_{j \to \infty} \Mass(N_j^-) \leq 4 \pi g.
  \end{equation}
  Moreover we have the following convergence in total variation:
  \begin{equation}
    \label{eq:convergence-total-variation}
    \left| \| N_j^+\| + \| N_j^-\| - \| (\Lambda_j)_\# \DBrack{\Sigma} \| \right|(V_2(\R^4)) \xrightarrow{j \to \infty} 0.
  \end{equation}
  \begin{proof}
    Since $\mathcal{P}_{-\tilde \rho_j}(\Lambda_{\tilde K_j}(\tilde P_j(x))) = \Lambda_j(x)$ whenever $x \in \Sigma_j^+$, we have the following identities:
    \begin{equation}
      \label{eq:Nj-plus}
      \begin{split}
      N_j^+
      &= (\mathcal{P}_{-\tilde \rho_j} \circ \Lambda_{\tilde K_j})_\# \DBrack{\tilde P_j(\Sigma_j^+)} + (\mathcal{P}_{-\tilde \rho_j} \circ \Lambda_{\tilde K_j})_\# \DBrack{\del \tilde K_j \setminus \tilde P_j(\Sigma_j^+)} \\
      &= (\Lambda_j)_\# \DBrack{\Sigma_j^+} + (\mathcal{P}_{-\tilde \rho_j} \circ \Lambda_{\tilde K_j})_\# \DBrack{\del \tilde K_j \setminus \tilde P_j(\Sigma_j^+)}
      \end{split}
    \end{equation}
    and
    \begin{equation}
      \label{eq:Nj-minus}
      N_j^-
      = -(\Lambda_{j})_\# \DBrack{\Sigma} + N_j^+
      = -(\Lambda_{j})_\# \DBrack{\Sigma \setminus \Sigma_j^+} + (\mathcal{P}_{-\tilde \rho_j} \circ \Lambda_{\tilde K_j})_\# \DBrack{\del \tilde K_j \setminus \tilde P_j(\Sigma_j^+)}.
    \end{equation}
    On one hand, it is clear that
    \[
      \Mass(N_j^+)
      = \Mass\left( (\Lambda_{\tilde K_j})_\# \DBrack{\del \tilde K_j} \right)
      \leq \AG(\tilde K_j)
      = 4\pi + o_j(1)
    \]
    by \cref{lem:area-gauss-convex-4pi}.
    On the other hand, using \eqref{eq:pushforward-gauss-measure} and \cref{prop:convex-sets},
    \begin{equation}
      \label{eq:estimate-extra-mass}
      \begin{split}
        \Mass\left((\Lambda_{\tilde K_j})_\# \DBrack{\del \tilde K_j \setminus \tilde P_j(\Sigma_j^+)} \right)
  &\leq \mu_{\tilde K_j}\left(\del \tilde K_j \setminus \tilde P_j(\Sigma_j^+) \right)
  = \mu_{\tilde K_j}(\del \tilde K_j) - \mu_{\tilde K_j}(\tilde P_j(\Sigma_j^+) ) \\
  &= \AG(\tilde K_j) - \mu_j(\Sigma_j^+) 
  \leq (4\pi + o_j(1)) - (4\pi - o_j(1)) \xrightarrow{j \to \infty} 0.
      \end{split}
    \end{equation}
    This gives
    \begin{align*}
      \Mass(N_j^-)
  &\leq \Mass\left((\Lambda_{j})_\# \DBrack{\Sigma \setminus \Sigma_j^+}\right) + \Mass\left((\mathcal{P}_{-\tilde \rho_j})_\# (\Lambda_{\tilde K_j})_\# \DBrack{\del \tilde K_j \setminus \tilde P_j(\Sigma_j^+)}\right) \\
  &\leq \mu_j(\Sigma \setminus \Sigma_j^+) + \Mass\left((\Lambda_{\tilde K_j})_\# \DBrack{\del \tilde K_j \setminus \tilde P_j(\Sigma_j^+)}\right) \\
  &\leq \mu_j(\Sigma) - \mu_j(\Sigma_j^+) + o_j(1)
  \leq 4\pi(1+g) - 4\pi + o_j(1)
  = 4\pi g + o_j(1).
    \end{align*}
    For the last claim, note that since $\| (\Lambda_j)_\# \DBrack{\Sigma} \| = \| (\Lambda_j)_\# \DBrack{\Sigma_j^+} \| + \| (\Lambda_j)_\# \DBrack{\Sigma \setminus \Sigma_j^+} \|$,
    \begin{align*}
      \left| \| N_j^+\| + \| N_j^-\| - \| (\Lambda_j)_\# \DBrack{\Sigma} \| \right|
      &\leq \left| \| N_j^+\| - \| (\Lambda_j)_\# \DBrack{\Sigma_j^+} \| \right|
      + \left| \| N_j^-\| - \| (\Lambda_j)_\# \DBrack{\Sigma \setminus \Sigma_j^+} \| \right| \\
      &\leq 2 \| (\mathcal{P}_{-\tilde \rho_j} \circ \Lambda_{\tilde K_j})_\# \DBrack{\del \tilde K_j \setminus \tilde P_j(\Sigma_j^+)} \| 
    \end{align*}
    thanks to \eqref{eq:Nj-plus} and \eqref{eq:Nj-minus}, and we conclude using \eqref{eq:estimate-extra-mass}.
  \end{proof}
\end{lemma}

\subsection{Proof of the compactness theorems}

We are finally ready to prove the main theorems of this paper.

\begin{proof}[Proof of \cref{thm:gauss-main-legendrian}]
  Recall that $[\pi_\# (\Lambda_j)_\# \DBrack{\Sigma}] = [(G_j)_\# \DBrack{\Sigma}] = (1-g, 1-g) \in \Z^2 \cong H_2(\Grass)$ (\cref{prop:degree-gauss-map}). Since the sets $\del \tilde K_j$ are spheres, the currents associated to their Gauss maps have $[\pi_\# N_j^+] = (1, 1)$, and therefore $[\pi_\# N_j^-] = (g, g)$ in $H_2(\Grass)$. Since $\Mass(N_j^+)$ and $\Mass(N_j^-)$ are uniformly bounded, after passing to another subsequence if necessary we have that
  \begin{equation}
    \label{eq:weak-limit-gauss}
    N_j^+ \weakto N^+
    \qquad \text{and} \qquad
    N_j^- \weakto N^-
  \end{equation}
  for Legendrian integral cycles $N^\pm$ in $V_2(\R^4)$. Moreover the Lagrangian cycles $T^\pm := \pi_\# N^\pm$ satisfy $[T^+] = (1, 1)$ and $[T^-] = (g, g)$ and, by \eqref{eq:decomp-mass} and the lower semicontinuity of the mass,
  \begin{equation}
    \label{eq:lsc-mass}
    \Mass(T^+) \leq \Mass(N^+) \leq 4\pi
    \qquad \text{and} \qquad
    \Mass(T^-) \leq \Mass(N^-) \leq 4\pi g.
  \end{equation}
It is well known that all of these homology classes in $\S^2_+(1/\sqrt{2}) \times \S^2_-(1/\sqrt{2})$ are calibrated by the K\"ahler form $\omega_+ + \omega_-$, therefore if $T$ is an integral cycle in the homology class $(k, k)$, for $k \geq 0$,
\[
  \Mass(T)
  \geq \langle T, \omega_+ + \omega_- \rangle
  = k \left( \int_{\S^2_+} \omega_+ + \int_{\S^2_-} \omega_- \right)
  = k \left( 2\pi + 2\pi \right) = 4\pi k
\]
with equality if and only if $T$ is holomorphic. Thus \eqref{eq:lsc-mass} is an equality and $T^+$ and $T^-$ are holomorphic $2$-dimensional integral currents. Now it is well known (see for example \cite{riviere-tian-annals} for a much more general statement) that $T^+$ and $T^-$ are represented by holomorphic curves with at most finitely many singularities $(z_1, w_1), \cdots, (z_p, w_p)$. 

For simplicity we will omit the factor $1/\sqrt{2}$ and the isomorphism $I$ from \eqref{eq:isometry-I}. On small open sets $U \subset \S^2_+$ away from $\{ z_1, \ldots, z_p \}$ there exist holomorphic functions $u_0, u_1, \ldots, u_g : U \to \S^2_- \cong \CP^1$ such that
\[
  T^+ \res U \times \S^2_- = \DBrack{\graph(u_0|_U)}
\]
and
\[
  T^- \res U \times \S^2_- = \DBrack{\graph(u_1|_U)} + \cdots + \DBrack{\graph(u_g|_U)}.
\]
At any point $z \in U$, the differentials $\Dd u_k : \C \cong T_z\S^2_+ \to \C \cong T_{u_k(z)}\S^2_-$ are complex-linear, and by the Lagrangian condition on $T^\pm$ we have that $(\id, u_k)^*(\omega_+, -\omega_-) = 0$, so that $\det_\R \Dd u_k = \det_\R \id = 1$. Hence $\Dd u_k \in \mathsf{U}(1) = \SO(2)$, which means that $u$ is locally an orientation-preserving isometry, and therefore there exists $Q_k \in \SO(3)$ such that $u_k(z) = Q_k z$ for all $z \in U$.

We can repeat this argument on each open set away from $\{z_1, \ldots, z_p\}$ and, by unique continuation for holomorphic functions, get that the rotations $Q_k$ agree globally. This gives that
\[
  T^+ = \DBrack{\graph(Q_0)}
  \qquad \text{and} \qquad
  T^- = \DBrack{\graph(Q_1)} + \cdots + \DBrack{\graph(Q_g)}.
\]
For each $0 \leq k \leq g$, choose $p_k \in \S^3$ such that $Q_k = R_{p_k}$, and using \cref{prop:graph-rotation-plane} write $\graph(R_{p_k}) = \{ p_k \wedge v : v \in \S^3 \cap p_k^\perp \}$. \cref{lem:legendrian-lift} below with $U = \Grass$ and $\Gamma_{p_0}$ gives immediately that
\[
  N^+ = \DBrack{\{ (p_0 \cos \theta_0 + z \sin \theta_0, -p_0 \sin \theta_0 + z \cos \theta_0) : z \in \S^3 \cap p_0^\perp \}}
\]
for some $\theta_0 \in [0, \pi]$. For $N^-$, assume (by reordering the points) that $\{p_1, \ldots, p_{g'}\}$ are all the distinct points that appear in $\{ p_1, \ldots, p_g \}$, with $1 \leq g' \leq g$, and apply \cref{lem:legendrian-lift} to $N^-$ with $\Gamma_{p_k}$ and $U = \Grass \setminus \cup_{1 \leq i \leq g', i \neq k} \Gamma_{p_i}$ (notice that $U \cap \Gamma_{p_k}$ is just a sphere with a finite set of points removed). This gives that
  \[
    N^- \res \pi^{-1}(\Gamma_{p_k})
    = N^- \res \pi^{-1}(U)
    = \sum_{i=1}^{m_k} \DBrack{\{ (p_k \cos \theta_{k,i} + z \sin \theta_{k,i}, -p_k \sin \theta_{k,i} + z \cos \theta_{k,i}) : z \in \S^3 \cap p_k^\perp \}},
  \]
  because
  \[
    \Mass(N^- \res \pi^{-1}(U \setminus \Gamma_{p_k}))
    = \Mass(N^- \res \pi^{-1}(\Grass \setminus \cup_{1 \leq i \leq g'} \Gamma_{p_i}))
    = \Mass(T^- \res \Grass \setminus \cup_{1 \leq i \leq g'} \Gamma_{p_i})
    = 0
  \]
  thanks to the equality in \eqref{eq:lsc-mass}. Since $\Gamma_{p_i} \cap \Gamma_{p_k}$ consists of just two points for $1 \leq i < k \leq g'$, we may add up these contributions and obtain
  \[
    N^-
    = \sum_{k=1}^{g'} N^- \res \pi^{-1}(\Gamma_{p_k})
    = \sum_{i=1}^{g} \DBrack{\{ (p_i \cos \theta_{i} + z \sin \theta_{i}, -p_i \sin \theta_{i} + z \cos \theta_{i}) : z \in \S^3 \cap p_i^\perp \}}
  \]
  for some $\theta_1, \ldots, \theta_g \in [0, \pi]$.

  We also have that $\| N_j^\pm \| \weakto \| N^\pm \|$, because no mass is lost at the limit, and therefore, denoting by $\Pi : V_2(\R^4) \subset \S^3 \times \S^3 \to \S^3$ the projection onto the first factor, also $\Pi_\# \| N_j^\pm \| \weakto \Pi_\# \| N^\pm \|$ as measures on $\S^3$. But using \eqref{eq:convergence-total-variation} and the fact that the supports of $\Pi_\# \| (\Lambda_j)_\# \DBrack{\Sigma}\|$ converge in the Hausdorff distance to $S$, it is clear that $\Pi_\# (\|N^+\| + \|N^-\|)$ must be supported in $S$ and therefore the spheres $\del B_{\theta_i}(p_i)$ are either points or equal to $S$.

  In view of \eqref{eq:convergence-total-variation}, to finish the proof of \cref{thm:gauss-main-legendrian} we just need to show that $\del B_{\theta_0}(p_0) = S$ and $\theta_i = 0$ for $i \geq 1$. We will therefore assume that $S$ is nondegenerate---otherwise there is nothing to prove. It is easy to see that $\spt \Pi_\# \| N^+ \| = \del B_{\theta_0}(p_0)$ must be equal to $S$.
  Now we claim that all the other spheres must be points, or equivalently, that if we write
  \[
    \Pi_\# N^- = q \DBrack{S}
  \]
  for $q \in \N$, then $q$ must be zero. Since
  \[
    (\phi_j)_\# \DBrack{\Sigma}
    = \Pi_\# (\Lambda_j)_\# \DBrack{\Sigma}
    = \Pi_\# N_j^+ - \Pi_\# N_j^-
    \weakto (1-q) \DBrack{S},
  \]
  it is enough to show that $\phi_j(\Sigma)$ lies in the same homology class as $S$ in a tubular neighborhood $B_\eps(S)$, for $\eps > 0$ small. But this is easy: for example, take a point $x_j \in \Sigma$ such that $\phi_j(\Sigma)$ is contained in $\S^3 \setminus (B^+ \cup B^-)$ for two balls $B^\pm$ that touch at $\phi_j(x_j)$ and such that $\S^3 \setminus (B^+ \cup B^-) \subset B_\eps(S)$, as in the proof of \cref{cor:convergence-round-sphere}. Then observe that, if $\sigma : \S^3 \setminus (B^+ \cup B^-) \to \del B^+$ denotes the nearest point projection, $\sigma$ is smooth and $\phi_j(x_j)$ is a regular point and the only preimage of itself, thus $\sigma \circ \phi_j$ has degree one and the claim follows.
\end{proof}

\begin{lemma}
  \label{lem:legendrian-lift}
  Suppose that $U \subset \Grass$ is an open set, $N$ is a Legendrian integral current in $V_2(\R^4)$ such that $(\del N) \res \pi^{-1}(U) = 0$ and $\spt (\pi_\# N) \cap U \subset \Gamma_p$, where $\Gamma_p = \{ p \wedge v : v \in \S^3 \cap p^\perp \} \subset \Grass$ for some $p \in \S^3$. Assume moreover that $U \cap \Gamma_p$ is connected. Then there exist $m \in \N$ and angles $\theta_1, \ldots, \theta_m \in [0, \pi]$ such that
  \[
    N \res \pi^{-1}(U) = \sum_{i=1}^m \DBrack{\{ (p \cos \theta_i + z \sin \theta_i, -p \sin \theta_i + z \cos \theta_i) : z \in \S^3 \cap p^\perp \}} \res \pi^{-1}(U).
  \]
  \begin{proof}
    We define the function $f : V_2(\R^4) \to \R$ by $f(a, b) = \langle a, p \rangle$. Since $N$ is Legendrian, $\langle \nabla^N a, b \rangle = 0$ $\|N\|$-almost everywhere, and also $\langle \nabla^N a, a \rangle = 0$ since $|a| = 1$. On the other hand, for every $(a, b) \in \spt N \cap \pi^{-1}(U) \subset \pi^{-1}(\Gamma_p)$ we have that $a \wedge b = p \wedge v$ for some $v \in \S^3 \cap p^\perp$, thus $p$ is a linear combination of $a$ and $b$ and it follows that $\nabla^N f = \langle \nabla^N a, p \rangle = 0$.

    For simplicity identify $p \wedge v \in \Gamma_p$ with $v \in \S^2$. We will now use a standard slicing technique: let $\alpha$ be a $1$-form on $\Gamma_p$ compactly supported in $U \cap \Gamma_p$ and $\psi \in C^\infty_c(\R)$, and consider the $0$-dimensional slices $\langle N, \pi, v \rangle$ for $v \in \S^2$, supported in the circle $\pi^{-1}(p \wedge v)$. We have that
    \begin{align*}
      \int \langle N, \pi, v \rangle(f^*\psi) \, \dd \alpha(v)
      &= N\left( (f^*\psi) \wedge \dd \pi^* \alpha \right)
      = -N\left( \dd ((f^*\psi) \wedge \pi^* \alpha) \right) + N\left( \dd (\psi \circ f) \wedge \pi^* \alpha) \right) \\
      &= -\del N\left( (f^*\psi) \wedge \pi^* \alpha \right)
      = 0
    \end{align*}
    because $N$ has no boundary in $\pi^{-1}(U)$ and $\nabla^N f \equiv 0$. This means that
    \[
      \int \dd \left( \langle N, \pi, v \rangle(f^*\psi) \right) \wedge \alpha(v)
      = \int \langle N, \pi, v \rangle(f^*\psi) \, \dd \alpha(v)
      = 0
    \]
    for every $\alpha$ and therefore, since $U \cap \Gamma_p$ is connected, $f_\# \langle N, \pi, v \rangle(\psi) = \langle N, \pi, v \rangle(f^*\psi)$ is constant (up to a set of measure zero) for each $\psi$ and hence the $0$-current $f_\# \langle N, \pi, v \rangle$ is independent of $v$. If we write $\langle N, \pi, v \rangle = \sum_{i=1}^{m(v)} \DBrack{(p \cos \theta_i(v) + v \sin \theta_i(v), -p \sin \theta_i(v) + v \cos \theta_i(v))}$, then $f_\# \langle N, \pi, v \rangle = \sum_{i=1}^{m(v)} \DBrack{\cos \theta_i(v)}$ and we deduce that $m$ is a constant and the $\theta_i$ are constants too (almost everywhere).
    The lemma follows since the vertical slices characterize a Legendrian current completely: any $2$-form $\beta$ on $\pi^{-1}(U) = \{ (p \cos \theta + v \sin \theta, -p \sin \theta + v \cos \theta) : v \in U, \theta \in \S^1 \}$ can be written in local coordinates as $h(v, \theta) \dd v^1 \wedge \dd v^2 + g_1(v, \theta) \dd v^1 \wedge \dd \theta + g_2(v, \theta) \dd v^2 \wedge \dd \theta$ and then, since $\nabla^N \theta = 0$,
    \[
      \pushQED{\qed}
      N (\beta)
      = N (h(v, \theta) \dd v^1 \wedge \dd v^2)
      = \int_U \langle N, \pi, v \rangle (h(v, \cdot)) \, \dd v.
      \qedhere
    \]
  \end{proof}
\end{lemma}

\begin{proof}[Proof of \cref{thm:gauss-main}]
  First we obtain a subsequence, a sphere $S$ and points $p_0, p_1, \ldots, p_g$ satisfying the conclusions of \cref{thm:gauss-main-legendrian}. Now \eqref{eq:convergence-cycles-1} and \eqref{eq:convergence-cycles-2} are clear from \eqref{eq:convergence-legendrian-cycles}, and by applying $\Pi_\#$ on \eqref{eq:convergence-associated-measures} we obtain
  \[
    \gamma_j
    = \Pi_\# \| (\Lambda_j)_\# \DBrack{\Sigma} \|
    \xrightharpoonup{\; j \to \infty \;} \Pi_\# (\Haus^2 \res \Gamma_{p_0}) + \sum_{k=1}^g \Pi_\# (\Haus^2 \res \Gamma_{p_k})
    = \gamma_S + \sum_{k=1}^g 4\pi \delta_{p_k}
  \]
  as measures on $\S^3$, which is \eqref{eq:convergence-measures-s3}. The varifold convergence \eqref{eq:convergence-varifolds} comes from the separate varifold convergence of
  \[
    \mathbf{v}(\pi_\# N_j^+) \weakto \mathbf{v}(\Gamma_{p_0})
    \qquad \text{and} \qquad
    \mathbf{v}(\pi_\# N_j^-) \weakto \mathbf{v}(\Gamma_{p_1}) + \cdots + \mathbf{v}(\Gamma_{p_g})
  \]
  (which can be justified, for example, using \eqref{eq:convergence-legendrian-cycles}, the convergence of the total masses of $\pi_\# N_j^\pm$, and Reshetnyak's continuity theorem~\cite[Theorem~10.3]{rindler-book} treating currents as vector-valued measures), together with \eqref{eq:convergence-total-variation}.
\end{proof}
\printbibliography

\end{document}